\documentclass[12pt]{amsart}
\usepackage{a4wide,enumerate,xcolor,graphicx}
\usepackage{amsmath}
\usepackage{scrextend} 
\usepackage{float}
\usepackage{subcaption}
\allowdisplaybreaks

\let\pa\partial
\let\na\nabla
\let\eps\varepsilon
\newcommand{\N}{{\mathbb N}}
\newcommand{\R}{{\mathbb R}}

\newcommand{\diver}{\operatorname{div}}

\newtheorem{theorem}{Theorem}
\newtheorem{lemma}[theorem]{Lemma}

\newtheorem{remark}[theorem]{Remark}

\begin{document}

\title[A diffusion system for biofilm growth]{Existence analysis for a reaction-diffusion
Cahn--Hilliard-type system with degenerate mobility and singular potential modeling 
biofilm growth} 

\author[C. Helmer]{Christoph Helmer}
\address{Institute of Analysis and Scientific Computing, Vienna University of  
	Technology, Wiedner Hauptstra\ss e 8--10, 1040 Wien, Austria}
\email{christoph.helmer@tuwien.ac.at} 

\author[A. J\"ungel]{Ansgar J\"ungel}
\address{Institute of Analysis and Scientific Computing, Vienna University of  
	Technology, Wiedner Hauptstra\ss e 8--10, 1040 Wien, Austria}
\email{juengel@tuwien.ac.at} 

\date{\today}

\thanks{The authors acknowledge partial support from   
the Austrian Science Fund (FWF), grants P33010 and F65.
This work has received funding from the European 
Research Council (ERC) under the European Union's Horizon 2020 research and 
innovation programme, ERC Advanced Grant no.~101018153.} 

\begin{abstract}
The global existence of bounded weak solutions to a diffusion system modeling biofilm growth is proven. The equations consist of a reaction-diffusion equation for the substrate concentration and a fourth-order Cahn--Hilliard-type equation for the volume fraction
of the biomass, considered in a bounded domain with no-flux boundary conditions. The main difficulties are coming from the degenerate diffusivity and mobility, the singular potential arising from a logarithmic free energy, and the nonlinear reaction rates.
These issues are overcome by a truncation technique and a Browder--Minty trick to identify the weak limits of the reaction terms.
The qualitative behavior of the solutions is illustrated by numerical experiments in one space dimension, using a BDF2 (second-order backward Differentiation Formula) finite-volume scheme.
\end{abstract}

\keywords{Biofilms, reaction-diffusion equation, Cahn--Hilliard equation, degenerate
mobility, singular potential, existence of solutions, logarithmic free energy.}
 
\subjclass[2000]{35K35, 35K65, 35K67, 35Q92, 92C17.}

\maketitle


\section{Introduction}

Biofilms are prevalent in nature and occur, for instance, 
in lakes, on rocks, and in sediments.
They play an important rule in medicine, where they attach surfaces of biomedical devices 
like catheters, and in wastewater treatment, where they convert organic matter in the water
into bacterial biomass. Biofilms consist of microorganisms that are
embedded in extracellular polymeric stubstances (EPS), which are produced by the bacteria
within the biofilm. In this paper, we analyze a variant of the model derived in \cite{WaZh12}
from kinetic equations. The model consists of a reaction-diffusion equation for the
substrate concentration and a Cahn--Hilliard-type equation for the volume fraction
of the biomass, composed of the EPS and bacteria. 
The particular feature of this model is that it contains a degenerate
diffusivity and mobility, a singular potential, and nonlinear production rates.

\subsection{Model setting}

The dynamics of the biofilm is given by the volume fraction
of the biomass $u(x,t)$ and the substrate concentration $v(x,t)$:
\begin{align}
  & \pa_t v - \diver((1-u)\na v) = g(u,v), \label{1.v} \\
	& \pa_t u - \diver(M(u)\na\mu) = h(u,v), \label{1.u} \\
	& \mu = -\Delta u + f'(u)\quad\mbox{in }\Omega,\ t>0, \label{1.mu}
\end{align}
where $\Omega\subset\R^d$ ($d\ge 1$) is a bounded domain.
Denoting by $u_s$ the solvent concentration, we impose the
volume-filling condition $u_s+u=1$ \cite[Section 5.2]{WaZh10}.
Equations \eqref{1.v}--\eqref{1.mu} are scaled, and we have set
the scaled physical parameters equal to one; see Section \ref{sec.scaling}
for the physical values. The initial and boundary conditions read as
\begin{align}
  & u(0) = u^0, \quad v(0) = v^0\quad\mbox{in }\Omega, \label{1.ic} \\
	& (1-u)\na v\cdot\nu = M(u)\na\mu\cdot\nu = \na u\cdot\nu = 0\quad\mbox{on }\pa\Omega,\ t>0.
	\label{1.bc}
\end{align}
 
The diffusion of the concentration vanishes if there is no solvent,
which means that equation \eqref{1.v} is degenerate with a nonstandard degeneracy.  
We suppose that the dynamics of the biomass is a gradient flow with
the chemical potential $\mu$ and the mobility $M(u)$. The mobility vanishes
if the biomass or the solvent vanish, $M(0)=M(1)=0$, and we choose
\begin{equation}\label{1.M}
  M(u) = u(1-u). 
\end{equation}
More general choices are possible; see Remark \ref{rem.gener}.
The chemical potential $\mu=-\Delta u+f'(u)$ is the variational derivative of the 
phase-separation gradient energy and the Flory--Huggins mixing free energy \cite{Flo42,Hug41}, given by its (nonconvex) density
\begin{equation}\label{1.FH}
	f(u) = \frac{1}{N}u\log u + (1-u)\log(1-u) + \lambda u(1-u).
\end{equation}
where $N>0$ is the generalized polymerisation index and $\lambda>0$ the
Flory--Huggins mixing parameter. The reaction terms are given by
\begin{equation}\label{1.gh}
  g(u,v) = -ug_0(v), \quad h(u,v) = u(1-u)h_0(v),
\end{equation}
where $g_0$ and $h_0$ are continuous functions. Examples are $g_0(v)=v$ and
$h_0(c)=v/(K+v)$ with $K>0$ \cite{WaZh10}.
This means that the substrate is consumed by the EPS, such that the consumption rate $g(u,v)$
is proportional to both the substrate concentration and the biomass fraction, and the polymer production rate $h_0(u)$ is modeled by Monod kinetics with half-saturation rate $K$. Here, we allow for more general reaction functions; see Assumption (A3) below.

Compared to the model in \cite{WaZh12}, we have modified the equations.
First, we neglected the velocities of the biomass and the solvent. 
Assuming that both are given by the same average velocity, 
it may be a given function or be determined
by the incompressible Navier--Stokes equations, see \cite[(5)--(6)]{Zha12}.
Our analysis works if we add a given velocity with bounded divergence. 
Second, we added the solvent fraction $u_s=1-u$ as a factor to the production rate and the mobility in equation \eqref{1.u}.
This is needed to guarantee the bound $u\le 1$ and to derive the entropy inequality associated to the system (see Section \ref{sec.main}
for details). Third, we neglect the elastic energy which simplifies the definition of the chemical potential.
Fourth, and most importantly, we have simplified the time derivative in equation \eqref{1.v} for the solvent concentration. Wang and Zhang \cite{WaZh12} suggested the two-phase equation
$\pa_t(u_sv)-\diver(u_s\na v)=g(u,v)$. However, the derivative $\pa_t(u_sv)$ introduces another degeneracy at $u_s=0$, which we are not able to treat.
A two-phase model with such a degeneracy was analyzed in \cite{JRZ19},
but in this work, the bounds for the volume fractions are a consequence
of the assumptions on the nonlinearities, which do not hold in the present situation.

\subsection{State of the art}

In the literature, many models for biofilm growth have been presented. One of the first models
was suggested in \cite{WaGu86}, consisting of a transport equation for the biofilm mass
and a differential equation for the biofilm thickness. This model, extended 
to multispecies biofilms with an equation for the free boundary, was analyzed in 
\cite{DAFr11} and refined in \cite{DFLM19} (to describe biofilm attachment). 
A different approach, based on diffusion equations
coupled to fluiddynamical models, was proposed in \cite{EPL01} and mathematically analyzed in
\cite{EEZ02}; also see the extensions in \cite{EEWZ14,ESE17}
and the numerical analysis in \cite{DJZ21,HJZ23}. 
The model of \cite{EPL01} describes the dynamics of the biomass density and nutrient
concontration, coupled with the incompressible homogeneous Navier--Stokes equations.
The mobility in the biomass equation is assumed to vanish if the biomass vanishes and blows up
if the biomass reaches its maximal value. In this way, the existence of a ``sharp front''
of biomass at the fluid/solid transition and significant biomass spreading close to the
maximum biomass value can be achieved. Another idea is to formulate the biofilm growth
as a free-boundary problem, modeling an incompressible viscous Stokes fluid in one phase
and a mixture of viscous fluid and the polymeric network in the other phase \cite{FHX14}.
Another free-boundary problem was suggested in \cite{CoKe04}, taking into account
surface forces, frictional drag generated by the EPS, hydrostatic pressure, and 
osmotic pressure that is modeled by the potential $f'(u)$ 
in the framework of the Flory--Huggins theory (see \eqref{1.FH}).

This approach was extended in \cite{ZCW08a,ZCW08b} by assuming that the biomass is driven
by the chemical potential given by a free energy density that includes the Flory--Huggins
mixing term and a gradient energy density. Then the diffusion equation for the biomass
becomes of fourth order and is similar to the Cahn--Hilliard equation, which was introduced
to study phase separation in binary alloys \cite{CaHi58}. Since fourth-order equations
generally do not allow for a maximum principle, the assumption that the mobility
vanishes at the minimal and maximal value of the mass variable guarantees
lower and upper bounds. The first existence analysis
of Cahn--Hilliard equations was given in \cite{Yin92} in one space dimension and in
\cite{ElGa96} in several space dimensions. Most of the analytical results on the 
Cahn--Hilliard equations do not contain reaction terms. Moreover, if reaction terms
are included in the Cahn--Hilliard model, nondegenerate mobilities are required;
see, e.g., \cite{AKK11,CMZ14,GaLa16}. When the gradient term in the free energy is replaced
by a nonlocal spatial interaction energy, degenerate mobilities (and singular potentials) 
can be treated \cite{Fri20,IoMe18}. 
{Up to our knowledge, there are only few papers which consider degenerate mobilities, singular potentials, and nonlinear reaction terms.
In \cite{EGN21}, the authors consider a degenerate mobility and a singular potential combined with reaction terms. In contrast to our work, upper bounds for the variables cannot be proved due to the choice of the potential. Furthermore, our degeneracy in the coupled reaction--diffusion equation is different, which causes additional difficulties. Upper bounds have been proved in \cite[Section 7.5]{E20}, but the reaction terms contain the chemical potential.}

\subsection{Main result and key ideas}\label{sec.main}

We impose the following assumptions:

\begin{itemize}
\item[(A1)] Domain: $\Omega\subset\R^d$ $(d\ge 1)$ is a bounded domain with Lipschitz 
continuous boundary. Set $\Omega_T=\Omega\times(0,T)$.
\item[(A2)] Initial data: $u^0\in H^1(\Omega)$ satisfies $0<u_*\le u^0\le u^*$ in $\Omega$ 
for some $u_*,u^*>0$ and $v^0\in L^2(\Omega)$ satisfying $0\le v^0\le 1$ in $\Omega$.
\item[(A3)] Source terms: $g_0\in C^0([0,1])$ is nondecreasing and satisfies $g_0(0)=0$, 
and $h_0\in C^1([0,1])$ is nondecreasing.
\end{itemize}

Our main result is the global existence of bounded weak solutions.

\begin{theorem}[Global existence]\label{thm.ex}
Let Assumptions (A1)--(A3) hold. 
Then there exists a weak solution $(u,v)$ to \eqref{1.v}--\eqref{1.bc} with the
constitutive relations \eqref{1.M}--\eqref{1.gh}, satisfying
$0\le u\le 1$, $0\le v\le 1$ in $\Omega_T$,
\begin{align*}
  & u\in L^2(0,T;H^2(\Omega))\cap C^0([0,T];H^1(\Omega)), \\
	& (1-u)\na v,\ \pa_tu,\ \pa_t v\in L^2(0,T;H^1(\Omega)'),  
\end{align*}
and the weak formulation for all $\phi_1$, $\phi_2\in L^2(0,T;H^2(\Omega))$,
\begin{align*}
  \int_0^T\langle\pa_t v,\phi_1\rangle dt
	+ \int_0^T\langle (1-u)\na v,\na\phi_1\rangle dt
	&= \int_0^T\int_\Omega g(u,v)\phi_1 dxdt, \\
	\int_0^T\langle\pa_t u,\phi_2\rangle dt
	+ {\int_0^T\int_\Omega J\cdot\na \phi_2 dxdt}
	&= \int_0^T\int_\Omega h(u,v)\phi_2 dxdt,
\end{align*}
where $\langle\cdot,\cdot\rangle$ is the dual product between $H^1(\Omega)'$ and $H^1(\Omega)$, and {$J\in L^2(\Omega_T;\R^d)$. The expressions $(1-u)\na v$ and $J=-M(u)\na \mu$ are understood in the weak sense, i.e., for all $\xi,\chi\in L^2(0,T;H^1(\Omega;\R^d))$ with 
$\xi\cdot\nu=\chi\cdot\nu = 0$ on $\pa\Omega \times (0,T)$, it holds that 
\begin{align*}
    \int_0^T\langle (1-u)\na v,\xi\rangle dt
    &= \int_0^T\int_\Omega v(-\na u\cdot\xi
    + (1-u)\diver\xi)dx dt, \\
	\int_0^T\int_{\Omega} J \cdot \chi dx dt &= 
	\int_0^T\int_{\Omega}\big(\Delta u \diver(M(u)\chi)  
	+ M(u) f''(u) \nabla u \cdot \chi\big) dx dt. 
\end{align*}}
\end{theorem}
 
{Note that the weak formulation of $J$ is possible since $M(u)f''(u)$ is bounded.}
The proof of Theorem \ref{thm.ex} is based on a suitable approximation scheme, 
truncating the nonlinearities and using a Galerkin method similarly as in \cite{ElGa96}.
Uniform estimates are obtained from the energy and entropy equalities, proved
in Lemma \ref{lem.energy} for the sequence of approximate solutions, 
\begin{align}
  \frac{d}{dt}\int_\Omega\bigg(\frac12|\na u|^2 + f(u)\bigg)dx
	+ \int_\Omega M(u)|\na\mu|^2 dx &= \int_\Omega h(u,v)\mu dx, \label{1.energy} \\
	\frac{d}{dt}\int_\Omega\Phi(u)dx + \int_\Omega\big((\Delta u)^2 + f''(u)|\na u|^2\big)dx
	&= \int_\Omega h(u,v)\Phi'(u)dx, \label{1.entropy}
\end{align}
where $\Phi$ is defined by $\Phi''(u)=1/M(u)$ and $\Phi(1/2)=\Phi'(1/2)=0$.
{These identities can be obtained formally as follows. We multiply \eqref{1.u} by $\mu$ and observe that
$$
  \langle\pa_t u,\mu\rangle 
  = \langle\pa_t u,-\Delta u+f'(u)\rangle
  = \frac{d}{dt}\bigg(\frac12\int_\Omega|\na u|^2 + f(u)\bigg)dx
$$
to find the energy identity \eqref{1.energy}. Furthermore, a multiplication of \eqref{1.mu} by $\Phi'(u)$ and integration over $\Omega$ leads to the entropy identity \eqref{1.entropy} since,
because of $\Phi''(u)=1/M(u)$ and integration by parts,
\begin{align*}
  \int_\Omega M(u)\na\mu\cdot\na\Phi'(u)dx
  &= \int_\Omega M(u)\na(-\Delta u+f'(u))\cdot\na u\Phi''(u)dx \\
  &= \int_\Omega\big((\Delta u)^2 + f''(u)|\na u|^2)dx.
\end{align*}}%
The function $\Phi(u)$ can be interpreted as the thermodynamic entropy of the system, since a computation shows that, with $M(u)$ given by \eqref{1.M},
$$
  \Phi(u) = u\log u + (1-u)\log(1-u) + \log 2\ge 0\quad\mbox{for }0<u<1.
$$
Since $f''(u)|\na u|^2\ge -2\lambda|\na u|^2$ for $0 < u < 1$, the corresponding integral in \eqref{1.entropy} can 
be bounded by Gronwall's lemma and the energy bound \eqref{1.energy}.

The difficulty is to estimate the right-hand sides of \eqref{1.energy}--\eqref{1.entropy}. 
The term $h(u,v)$ contains the
factor $u(1-u)$ which cancels the singularity from $\Phi'(u)$, such that
$\int_\Omega h(u,v)\Phi'(u)dx$ is bounded. For the other integral, we include the
definition of $\mu$ and integrate by parts:
$$
  \int_\Omega h(u,v)\mu dx = \int_\Omega\big((1-2u){h_0(v)}|\na u|^2 
	+ u(1-u)h'_0(v)\na v\cdot\na u + h(u,v)f'(u)\big)dx.
$$
The last term is bounded since $h(u,v)$ cancels the singularity of the potential $f'(u)$.
The first term can be treated by Gronwall's lemma since $u$ is bounded. For the second term, we use Young's inequality:
$$
  \int_\Omega u(1-u)h'_0(v)\na v\cdot\na u dx
	\le C\int_\Omega|\na u|^2 dx + C\int_\Omega(1-u)|\na v|^2 dx,
$$
where we use the property $0\le v\le 1$. The last integral can be absorbed
by the energy bound for $v$:
\begin{equation}\label{1.substrate}
  \frac12\frac{d}{dt}\int_\Omega v^2 dx + \int_\Omega(1-u)|\na v|^2 dx
	= \int_\Omega g(u,v)v dx \le C.
\end{equation}

There is another difficulty: Because of the degeneracy in the equation for $v$,
we do not obtain an estimate for $\na v$ (see \eqref{1.substrate}) and therefore
we cannot expect strong convergence for (a subsequence of) 
the approximate solutions $(v_\delta)$ with $\delta>0$ being an approximation parameter,
but only weak* convergence in $L^\infty(\Omega_T)$. Surprisingly, the weak convergence
of $(v_\delta)$ is enough to pass to the limit $\delta\to 0$ in $(1-u_\delta)\na v_\delta$,
since this expression can be written as $\na((1-u_\delta)v_\delta)+v_\delta\na u_\delta$,
which converges weakly in the sense of distributions, since $(\na u_\delta)$ converges
strongly (up to a subsequence). However, the weak convergence is not sufficient to
perform the limit in the reaction rates. The idea is to use the duality of $H^1(\Omega)'$
and $H^1(\Omega)$ as well as a Minty-Browder trick. Indeed, since $h_0$ is nondecreasing, we have for $y\in C_0^\infty(\Omega_T)$,
\begin{align*}
  0 &\le \int_0^T\int_\Omega u_\delta(1-u_\delta)(v_\delta-y)(h_0(v_\delta)-h(y))dxdt \\
	&= \int_0^T\big\langle v_\delta-y,u_\delta(1-u_\delta)(h_0(v_\delta)-h_0(y))\big\rangle dt.
\end{align*}
(Observe that we need to truncate the factor $u_\delta(1-u_\delta)$, since we cannot expect
that $0\le u_\delta\le 1$; see Section \ref{sec.trunc}.)
By the Aubin--Lions lemma, $v_\delta\to v$ strongly in $L^2(0,T;H^1(\Omega)')$
and $u_\delta\to u$ strongly in $L^2(0,T;H^1(\Omega))$. 
Hence, a computation shows that the limit $\delta\to 0$ in the previous inequality leads to
$$
  0 \le \int_0^T\big\langle v-y,u(1-u)(h_1-h(y))\big\rangle dt,
$$
where $h_1$ is the weak $L^2(\Omega_T)$-limit of $(h_0(v_\delta))$. A Minty--Browder argument,
made precise in Lemma \ref{lem.gh}, shows that $h_1=h_0(v)$, implying that
$h(u_\delta,v_\delta)\rightharpoonup h(u,v)$ weakly in $L^2(\Omega_T)$. 

The paper is organized as follows. We formulate and prove the existence of a solution 
to a truncated regularized system in Section \ref{sec.approx}. We truncate the mobility and
the mixing free energy using the parameter $\delta$, add the regularization $\kappa>0$
in the equation for $v$ (because of the degeneracy), and introduce the Galerkin
dimension $L\in\N$. First estimates allow us to perform the limit $L\to\infty$.
Estimates uniform in $(\delta,\kappa)$ are derived from the energy and entropy
inequalities in Section \ref{sec.unif}. In Section \ref{sec.limit}, we pass to 
the limit $\delta=\kappa\to 0$. Finally, we present some numerical experiments
in one space dimension in Section \ref{sec.num} to compare our model with that one 
of \cite{WaZh12}.


\section{Existence for the approximate system}\label{sec.approx}

\subsection{Truncated regularized system}\label{sec.trunc}

We truncate the functions $M(u)$, $f(u)$, and the source terms. Let $\delta>0$
and set $[u]_+=\max\{0,u\}$ and $[u]_+^1=\min\{1,\max\{0,u\}\}$ for $u\in\R$. We
introduce for $u\in\R$
\begin{align*}
  M_\delta(u) &= \begin{cases}
	M(\delta) &\quad\mbox{if }u\le\delta, \\
	M(u) &\quad\mbox{if }\delta<u<1-\delta, \\
	M(1-\delta) &\quad\mbox{if }u\ge 1-\delta.
	\end{cases}
\end{align*}
Then $M_\delta(u)\ge M(u)$ for $u\in\R$. Furthermore, we set
$$
  D_+(u)=[1-u]_+^1. 
$$	
We approximate
the singular part $f_1(u)=N^{-1}u\log u+(1-u)\log(1-u)$ of the free energy by setting
\begin{align*}
			f_{1,\delta}(u) = \begin{cases} 
	\left( -f_1''(\delta)(\frac{3}{2}+\delta) + f_1'(\delta)\right) u + C_1(\delta), &\text{if } u \leq -2, \\
	f_1''(\delta)\big(\frac{1}{6} u^3+u^2+(\frac{1}{2}-\delta)u\big) + f_1'(\delta)u + C_2(\delta), &\text{if } -2 \leq u \leq -1, \\
	f_1(\delta)+f_1'(\delta)(u-\delta)+\frac{1}{2}f_1''(\delta)(u-\delta)^2, &\text{if } -1 \leq u \leq \delta, \\ 
	f_1(u), &\text{if } \delta < u < 1-\delta, \\
	f_1(1-\delta)+f_1'(1-\delta)(u-(1-\delta))+f_1''(1-\delta)(u-(1-\delta))^2, &\text{if } 1-\delta \leq u \leq 2, \\
	f_1''(1-\delta)\big(\frac{1}{6}u^3 + \frac{3}{2}u^2-(2+(1-\delta))u\big)+f_1'(1-\delta)u + C_3(\delta), &\text{if } 2 \leq u \leq 3, \\
	 \left((\frac{5}{2}-(1-\delta))f_1''(1-\delta)+f_1'(1-\delta)\right) u + C_4(\delta), &\text{if } u \geq 3,
	\end{cases}
\end{align*}
where 
\begin{align*}
	C_1(\delta) &= C_2(\delta) - \frac{4}{3} f_1''(\delta), \\
	C_2(\delta) &=  f_1(\delta)-f_1'(\delta)\delta +\frac{1}{2} f_1''(\delta)\bigg((1+\delta)^2-2\delta-\frac{2}{3}\bigg), \\
	C_3(\delta) &= f_1(1-\delta)+f_1'(1-\delta)(\delta-1) + f_1''(1-\delta)\bigg(\frac{1}{2} (\delta+1)^2 - 2\delta + \frac{4}{3}\bigg), \\
	C_4(\delta) &= C_3(\delta)-\frac{9}{2} f_1''(1-\delta).
\end{align*}
This means that 
\begin{align}\label{2.f1pp}
	f_{1,\delta}''(u) = \begin{cases} 
	0, \quad &\text{if } u < -2, \\
	f_1''(\delta)(u+2)	, \quad &\text{if } -2 \leq u \leq -1, \\
	f_1''(\delta), \quad &\text{if } -1 \leq u \leq \delta, \\ 
	f_1''(u), \quad &\text{if } \delta < u < 1-\delta, \\
	f_1''(1-\delta), \quad &\text{if } 1-\delta \leq u \leq 2, \\
	f_1''(1-\delta)(3-u),\quad &\text{if } 2 \leq u \leq 3, \\
	0, \quad &\text{if } u > 3,
	\end{cases}
\end{align}
and ensures that $|f_{1,\delta}'(u)| \leq C(\delta)$ for $u \in \R$. 

\begin{figure}\label{plotf}
\includegraphics[width=0.49\textwidth]{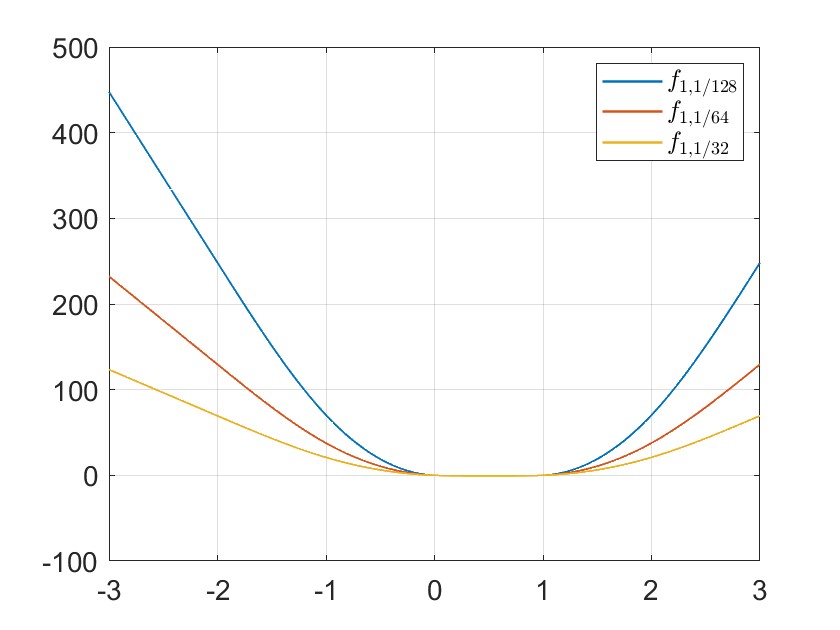}
\includegraphics[width=0.49\textwidth]{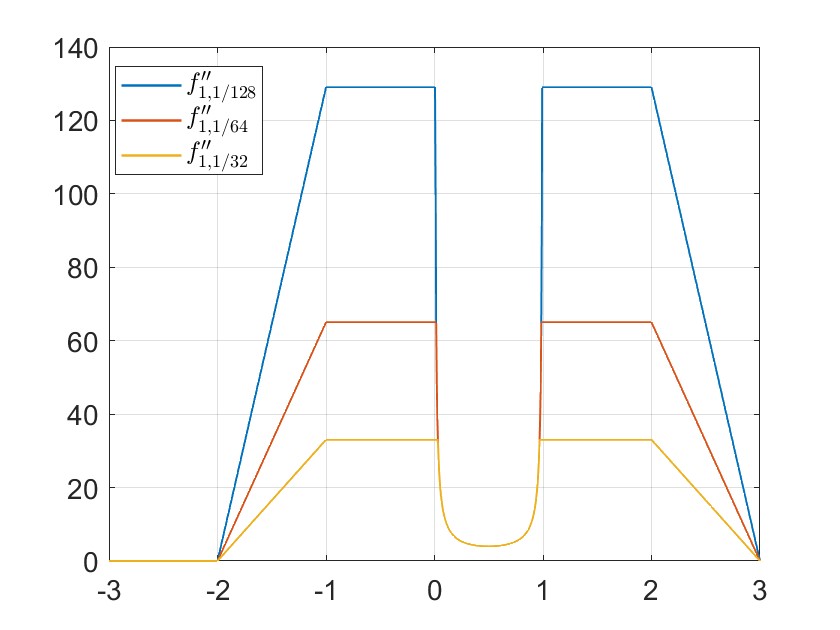}
\caption{Functions $f_{1,\delta}$ (left) and $f_{1,\delta}''$ (right) with $N = 1$.}
\end{figure}

The regular (nonconvex) part $f_2(u)=\lambda u(1-u)$ ($0\le u\le 1$) of the free energy is 
extended to $\R$ such that $|f_2(u)|\le C$ for $u\in\R$. Furthermore, we set
$f_\delta=f_{1,\delta}+f_2$, and this function is defined for all $u\in\R$.
We also need to truncate the source terms:
$$
  g_+(u,v) = -[u]_+^1g_0([v]_+^1), \quad
	h_+(u,v) = [u]_+[1-u]_+h_0([v]_+^1).
$$
Finally, let $\kappa>0$. We wish to find a solution to the truncated and regularized system
\begin{align}
  & \pa_t v - \diver(D_+(u)\na v) - \kappa\Delta v = g_+(u,v), \label{2.v} \\
	& \pa_t u - \diver(M_\delta(u)\na\mu) = h_+(u,v), \label{2.u} \\
	& \mu = -\Delta u + f'_\delta(u)\quad\mbox{in }\Omega,\ t>0, \label{2.mu}
\end{align}
subject to the initial conditions \eqref{1.ic} and the Neumann boundary conditions
\begin{equation}\label{2.bc}
  \na v\cdot\nu = \na\mu\cdot\nu = \na u\cdot\nu = 0\quad\mbox{on }\pa\Omega,\ t>0.
\end{equation}


\subsection{Galerkin approximation}\label{sec.galer}

To solve \eqref{1.ic}, \eqref{2.v}--\eqref{2.bc}, we use the Galerkin method 
(as in \cite{ElGa96}).
Let $(\phi_\ell)_{\ell\in\N}$ be the orthonormal eigenfunctions of the Laplace operator
with homogeneous Neumann boundary conditions. We can assume that $\lambda_1=0$ and
$\phi_1=\mbox{const}$. Let $L\in\N$. We wish to find solutions
$$
	v_L(x,t) = \sum_{\ell=1}^L A_\ell(t)\phi_\ell(x), \quad
  u_L(x,t) = \sum_{\ell=1}^L B_\ell(t)\phi_\ell(x), \quad
	\mu_L(x,t) = \sum_{\ell=1}^L C_\ell(t)\phi_\ell(x)
$$
to the finite-dimensional system
\begin{align}
	\int_\Omega\pa_t v_L\phi dx &= -\int_\Omega (D_+(u_L)+\kappa)\na v_L\cdot\na\phi dx
	+ \int_\Omega g_+(u_L,v_L)\phi dx, \label{2.Gv} \\
	  \int_\Omega\pa_t u_L\phi dx &= -\int_\Omega M_\delta(u_L)\na\mu_L\cdot\na\phi dx
	+ \int_\Omega h_+(u_L,v_L)\phi dx, \label{2.Gu} \\
	\int_\Omega\mu_L\phi dx &= \int_\Omega\na u_L\cdot\na\phi dx 
	+ \int_\Omega f'_\delta(u_L)\phi dx \label{2.Gmu}
\end{align}
for all $\phi\in\operatorname{span}(\phi_1,\ldots,\phi_L)$, with the initial conditions
$$
  v_L(0) = \sum_{\ell=1}^{L}(v^0,\phi_\ell)_{L^2(\Omega)}\phi_\ell dx, \quad
	u_L(0) = \sum_{\ell=1}^{L}(u^0,\phi_\ell)_{L^2(\Omega)}\phi_\ell dx.
$$
This gives an initial-value problem for a system of ordinary differential equations
for $(A_1,\ldots,A_L)$ and $(B_1,\ldots,B_L)$:
\begin{align*}
	\pa_t A_\ell &= -\int_\Omega([u_L]_+^1+\kappa)\na v_L\cdot\na\phi_\ell dx
	+ \int_\Omega g_+(u_L,v_L)\phi_\ell dx, \\
	  \pa_t B_\ell &= -\int_\Omega M_\delta(u_L)\na\mu_L\cdot\na\phi_\ell dx
	+ \int_\Omega h_+(u_L,v_L)\phi_\ell dx, \\
	C_{\ell} &= \int_\Omega\na u_L\cdot\na\phi_\ell dx 
	+ \int_\Omega f'_\delta(u_L)\phi_\ell dx\quad\mbox{for }\ell=1,\ldots,L,
\end{align*}
with the initial conditions $A_\ell(0)=(v^0,\phi_\ell)_{L^2(\Omega)}$ and
$B_\ell(0)=(u^0,\phi_\ell)_{L^2(\Omega)}$. As the right-hand side of this system
is continuous in $(A_1,\ldots,A_L)$ and $(B_1,\ldots,B_L)$, the Peano theorem
ensures the existence of a local solution. To extend this solution globally, 
we prove some a priori estimates.

\begin{lemma}[Energy estimate for the Galerkin approximation]\label{lem.enerL}\ \sloppy
There exists a constant $C(\delta)>0$ independent of $L$ such that for all $t\in(0,T)$,
\begin{align*}
  \frac12\|\na u_L(t)\|_{L^2(\Omega)}^2 &+ \int_\Omega f_\delta(u_L(t))dx
	+ \frac12 M(\delta)\int_0^t\|\na\mu_L\|_{L^2(\Omega)}^2 ds \\
	&\le \frac12\|\na u_L(0)\|_{L^2(\Omega)}^2 + \int_\Omega f_\delta(u_L(0))dx + C(\delta),
\end{align*}
and because of Assumption (A2), the right-hand side can be bounded independently of $L$.
\end{lemma}

\begin{proof}
We choose $\phi=\mu_L$ in \eqref{2.Gu} and $\phi=\pa_t u_L$ in \eqref{2.Gmu}:
\begin{align*}
  \int_\Omega\pa_t u_L\mu_L dx &= -\int_\Omega M_\delta(u_L)|\na\mu_L|^2 dx
	+ \int_\Omega h_+(u_L,v_L)\mu_L dx \\
	&\le -M(\delta)\int_\Omega|\na\mu_L|^2 dx + C\|\mu_L\|_{L^1(\Omega)}, \\
	\int_\Omega\mu_L\pa_t u_L dx &= \int_\Omega\na u_L\cdot\na\pa_t u_L dx
	+ \int_\Omega f'_\delta(u_L)\pa_t u_L dx \\
	&= \frac{d}{dt}\bigg(\frac12\int_\Omega|\na u_L|^2dx
	+ \int_\Omega f_\delta(u_L)dx\bigg),
\end{align*}
since $|h_+(u_L,v_L)|\le C$ because of our truncations. Here and in the following,
$C>0$ denotes a generic constant with values changing from line to line.
Equating both expressions and integrating over $(0,t)$ gives
\begin{align}\label{2.aux}
  \frac12\int_\Omega\|\na u_L(t)\|_{L^2(\Omega)}^2	&+ \int_\Omega f_\delta(u_L(t))dx
	\le \frac12\int_\Omega|\na u_L(0)|^2 dx + \int_\Omega f_\delta(u_L(0))dx \\
	&{}- M(\delta)\int_0^t\int_\Omega|\na\mu_L|^2dxds 
	+ C\int_0^t\|\mu_L\|_{L^1(\Omega)}ds. \nonumber 
\end{align}
The choice $\phi_1=\mbox{const.}$ in \eqref{2.Gmu} shows that
$$
  \bigg|\int_\Omega\mu_L dx\bigg| \le \int_\Omega |f'_\delta(u_L)|dx \le C(\delta).
$$
Set $\bar{\mu}_L=|\Omega|^{-1}\int_\Omega\mu_L dx$.
By the Poincar\'e--Wirtinger inequality, the previous estimate 
provides a bound for the $L^2(\Omega)$ norm of $\mu_L$:
$$
  \|\mu_L\|_{L^2(\Omega)} \le \|\mu_L-\bar{\mu}_L\|_{L^2(\Omega)}
	+ \|\bar{\mu}_L\|_{L^2(\Omega)} \le C_P\|\na\mu_L\|_{L^2(\Omega)} + C(\delta).
$$
Applying Young's inequality, we have
$$
  \int_0^t\|\mu_L\|_{L^1(\Omega)}ds
	\le C(\Omega)\int_0^t\|\mu_L\|_{L^2(\Omega)}ds
	\le \frac12M(\delta)\int_0^t\|\na\mu_L\|_{L^2(\Omega)}^2ds 
	+ C(\delta,\Omega,T).
$$
Inserting this estimate into \eqref{2.aux} finishes the proof.
\end{proof}

\begin{lemma}[Estimates for $u_L$ and $\mu_L$]\label{lem.estuL}
There exists $C(\delta)>0$ independent of $L$ such that
$$
  \|u_L\|_{L^\infty(0,T;H^1(\Omega))} + \|\mu_L\|_{L^2(0,T;H^1(\Omega))} \le C(\delta).
$$
\end{lemma}

\begin{proof}
The proof of Lemma \ref{lem.enerL} shows that $(\na\mu_L)$ and $(\mu_L)$ are bounded
in $L^2(\Omega_T)$ and that $(\na u_L)$ is bounded in $L^\infty(0,T;L^2(\Omega))$.
We choose $\phi_1=\mbox{const.}$ in \eqref{2.Gu}:
$$
  \frac{d}{dt}\int_\Omega u_L dx = \int_\Omega h_+(u_L,v_L)dx \le C(\Omega).
$$
Consequently, $\int_\Omega u_L(t)dx$ is uniformly bounded,
at least on finite time intervals.
This allows us to apply the Poincar\'e--Wirtinger inequality to deduce an $L^2(\Omega)$
bound for $u_L(t)$ uniformly in time.
\end{proof}

We also need a priori estimates for the substrate concentration.

\begin{lemma}[Estimates for $v_L$]\label{lem.estvL}
There exists $C(v^0)>0$ only depending on the initial datum $v^0$ such that
$$
  \|v_L\|_{L^\infty(0,T;L^2(\Omega))} + \|D_+(u_L)^{1/2}\na v_L\|_{L^2(\Omega_T)}
	+ \kappa^{1/2}\|\na v^L\|_{L^2(\Omega_T)} \le C(v^0).
$$
\end{lemma}

\begin{proof}
We choose the test function $\phi=v_L$ in \eqref{2.Gv} and take into account that $g_0(0)=0$:
$$
  \frac12\frac{d}{dt}\int_\Omega v_L^2 dx
	+ \int_\Omega (D_+(u_L)+\kappa)|\na v_L|^2 dx 
	= -\int_\Omega[u_L]_+^1g_0([v_L]_+^1) v_Ldx \le 0.
$$
An integration over $(0,T)$ yields the result.
\end{proof}

The uniform estimates for $u_L$ in $L^\infty(0,T;H^1(\Omega))$ and $v_L$ in
$L^\infty(0,T;L^2(\Omega))$ show that the coefficients $(A_\ell)$ and $(B_\ell)$
are bounded in $(0,T)$. Thus, we infer the global existence of solutions
to the Galerkin system \eqref{2.Gv}--\eqref{2.Gmu}. To pass to the limit $L\to\infty$,
we need an estimate for the time derivatives.

\begin{lemma}[Estimates for the time derivatives]\label{lem.timeL}
There exist $C_1(\delta)>0$ depending on $\delta$ and $C_2>0$ independent of $\delta$
such that
$$
  \|\pa_t u_L\|_{L^2(0,T;H^1(\Omega)')} \le C_1(\delta), \quad
	\|\pa_t v_L\|_{L^2(0,T;H^1(\Omega)')} \le C_2.
$$
\end{lemma}

\begin{proof}
Let $\psi\in L^2(0,T;H^1(\Omega))$ and let $\Pi_L\psi$ be the projection of $\psi$ 
on $\operatorname{span}(\phi_1,\ldots,\phi_L)$. We infer from \eqref{2.Gv} and
the bounds of Lemma \ref{lem.estvL} that
\begin{align*}
  \bigg|\int_0^T&\int_\Omega\pa_t v_L\psi dxdt\bigg|
	= \bigg|\int_0^T\int_\Omega\pa_t v_L\Pi_L\psi dxdt\bigg| \\
	&\le \int_0^T\|D_+(u_L)^{1/2}\|_{L^\infty(\Omega)}
	\|D_+(u_L)^{1/2}\na v_L\|_{L^2(\Omega)}\|\na\Pi_L\psi\|_{L^2(\Omega)}dt \\
	&\phantom{xx}{}+ \int_0^t\|g_+(u_L,v_L)\|_{L^2(\Omega)}\|\Pi_L\psi\|_{L^2(\Omega)}dt
	\le C_2\|\psi\|_{L^2(0,T;H^1(\Omega))}.
\end{align*}
Furthermore, using $M_\delta(u_L)\le C_M$ and the bounds of Lemma \ref{lem.estuL}, 
\begin{align*}
  \bigg|\int_0^T&\int_\Omega\pa_t u_L\psi dxdt\bigg|
	= \bigg|\int_0^T\int_\Omega\pa_t u_L\Pi_L\psi dxdt\bigg| \\
	&\le C_M\int_0^T\|\na\mu_L\|_{L^2(\Omega)}\|\na\Pi_L\psi\|_{L^2(\Omega)}dt \\
	&\phantom{xx}{}+ \int_0^T\|h_+(u_L,v_L)\|_{L^2(\Omega)}\|\Pi_L\psi\|_{L^2(\Omega)}dt
	\le C_1(\delta)\|\psi\|_{L^2(0,T;H^1(\Omega))}.
\end{align*}
This concludes the proof.
\end{proof}

The estimates of Lemmas \ref{lem.estuL}--\ref{lem.timeL} allow us to apply the
Aubin--Lions lemma \cite[Corollary 4]{Sim87} to find subsequences (not relabeled) 
such that, as $L\to\infty$,
\begin{align*}
  u_L\to u, \quad v_L\to v &\quad\mbox{strongly in }L^2(\Omega_T), \\
	u_L\rightharpoonup u,\quad v_L\rightharpoonup v, \quad \mu_L\rightharpoonup\mu 
	&\quad\mbox{weakly in }L^2(0,T;H^1(\Omega)), \\
	\pa_t u_L\rightharpoonup \pa_t u, \quad \pa_t v_L\rightharpoonup \pa_t v
	&\quad\mbox{weakly in }L^2(0,T;H^1(\Omega)').
\end{align*}
Since $M_\delta$, $D_+$, $f'_\delta$, $g_+$, and $h_+$ are bounded functions, we have
\begin{align*}
  & M_\delta(u_L)\to M_\delta(u), \quad D_+(u_L)\to D_+(u), \quad
	f'_\delta(u_L)\to f'_\delta(u), \\
	& h_+(u_L,v_L)\to h_+(u,v), \quad g_+(u_L,v_L)\to g_+(u,v)
	\quad\mbox{strongly in }L^2(\Omega_T).
\end{align*}
Thus, we can perform the limit $L\to\infty$ in the Galerkin system 
\eqref{2.Gv}--\eqref{2.Gmu}, which yields the existence of a solution
$(u,v,\mu)$ to
\begin{align}
  \int_0^t\langle\pa_t u,\phi_1\rangle ds
	&= -\int_0^t\int_\Omega M_\delta(u)\na\mu\cdot\na\phi_1 dxds
	+ \int_0^t\int_\Omega h_+(u,v)\phi_1 dxds, \label{2.appu} \\
	\int_0^t\langle\pa_t v,\phi_2\rangle ds
	&= -\int_0^t\int_\Omega (D_+(u)+\kappa)\na v\cdot\na\phi_2 dxds
  + \int_0^t\int_\Omega g_+(u,v)\phi_2 dxds, \label{2.appv} \\
	\int_0^t\int_\Omega\mu\phi_3 dxds &= \int_0^t\int_\Omega\na u\cdot\na\phi_3 dxds
	+ \int_0^t\int_\Omega f'_\delta(u)\phi_3 dxds \label{2.appmu}
\end{align}
for all $\phi_i\in L^2(0,T;H^1(\Omega))$, $i=1,2,3$, and all $0<t<T$, recalling that $\langle\cdot,\cdot\rangle$ is the dual product between $H^1(\Omega)'$ and $H^1(\Omega)$. 


\section{Uniform estimates}\label{sec.unif}

We need some estimates uniform in $\delta$ and $\kappa$ as well as lower and upper bounds to remove the truncation.

\begin{lemma}[Uniform estimates for $v$]\label{lem.estv}
There exists $C(v^0)>0$ only depending on $v^0$ such that
$$
  \|v\|_{L^\infty(0,T;L^2(\Omega))} + \|D_+(u)^{1/2}\na v\|_{L^2(\Omega_T)}
	+ \kappa^{1/2}\|\na v\|_{L^2(\Omega_T)} \le C(v^0).
$$
Furthermore, it holds that $0\le v(t)\le 1$ in $\Omega$ for $0<t<T$.
\end{lemma}

Because of the lower and upper bounds for $v$, we can remove the truncation in
$g_+(u,v)=-[u]_+^1g_0(v)$ and $h_+(u,v)=[u]_+[1-u]_+h_0(v)$. 

\begin{proof}
We start with the lower and upper bounds for $v$. We use the test function $[v]_-=\min\{0,v\}$
in \eqref{2.appv} and use the assumption $v(0)\ge 0$ in $\Omega$:
$$
  \int_\Omega[v(t)]_-^2 dx
	+ \int_0^t\int_\Omega(D_+(u)+\kappa)|\na[v]_-|^2 dxds
	= \int_0^t\int_\Omega g_+(u,v)[v]_- dxds = 0,
$$
since $g_0(0)=0$ implies that $g_+(u,v)[v]_-=-[u]_+^1g_0(0)[v]_-=0$.
This implies that $v(t)\ge 0$ in $\Omega$, $t>0$. The property $v(t)\le 1$ is proved
in a similar way using the test function $[v-1]_+$ and the fact that
$g_+(u,v)[v-1]_+=-[u]_+^1g_0(1)[v-1]_+ \le 0$.
The remaining estimates can be shown as in Lemma \ref{lem.estvL}.
\end{proof}

Next, we show some uniform estimates for $u$. For this, we introduce the entropy density
\begin{equation}\label{2.Phi}
  \Phi_\delta(u) = \int_{1/2}^u\int_{1/2}^s \frac{drds}{M_\delta(r)}\ge 0.
\end{equation}

\begin{lemma}[Energy and entropy estimates]\label{lem.energy}
There exists $C(T)>0$ independent of $\delta$ and $\kappa$ such that for all $t>0$
and all sufficiently small $\delta>0$,
\begin{align}\label{2.energy}
  \sup_{0<t<T}\int_\Omega\bigg(\frac12|\na u(t)|^2dx + f_\delta(u(t))\bigg)dx
	+ \int_0^T\int_\Omega M_\delta(u)|\na\mu|^2 dxds &\le C(T), \\
	\sup_{0<t<T}\int_\Omega\Phi_\delta(u(t))dx + \int_0^T\int_\Omega(\Delta u)^2dxds &\le C(T). \label{2.entropy}
\end{align}
\end{lemma}

Since $f_\delta$ is bounded from below (by construction), the energy inequality
provides uniform bounds for $u$. 

\begin{proof}
We first prove the energy inequality and then the entropy inequality.

{\em Step 1: Energy inequality.} 
We know from Section \ref{sec.galer} that $u\in L^\infty(0,T;H^1(\Omega))$ and
$\mu\in L^2(0,T;H^1(\Omega))$. Then we infer from the boundedness of $f'_\delta$ that
$\Delta u = f'_\delta(u)-\mu\in L^2(0,T;L^2(\Omega))$.
By elliptic regularity theory, $u\in L^2(0,T;H^2(\Omega))$. 
Moreover, $\na\Delta u = f''_\delta(u)\na u-\na\mu
\in L^2(\Omega_T)$, which implies that $u\in L^2(0,T;H^3(\Omega))$
(this regularity is not uniform in $(\delta,\kappa)$).
Consequently, $\Delta u\in L^2(0,T;H^1(\Omega))$ and
\begin{align*}
  0 &= \int_0^t\langle\pa_t u,\mu+\Delta u-f'_\delta(u)\rangle ds \\
	&= \int_0^t\langle\pa_t u,\mu\rangle ds - \frac12\int_\Omega(|\na u(t)|^2-|\na u(0)|^2)ds
	- \int_\Omega(f_\delta(u(t))-f_\delta(u(0))ds.
\end{align*}
On the other hand, we use $\phi_2=\mu\in L^2(0,T;H^1(\Omega))$
as a test function in \eqref{2.appu}:
$$
  \int_0^t\langle\pa_t u,\mu\rangle ds + \int_0^t\int_\Omega M_\delta(u)|\na\mu|^2 dxds
	= \int_0^t\int_\Omega h_+(u,v)\mu dx.
$$
This shows that, using the definition of $\mu$,
\begin{align}\label{2.aux2}
  \frac12\int_\Omega&|\na u(t)|^2 dx + \int_\Omega f_\delta(u(t))dx
	+ \int_0^t\int_\Omega M_\delta(u)|\na\mu|^2 dxds
	= \frac12\int_\Omega|\na u^0|^2 dx \\
	&{}+ \int_\Omega f_\delta(u^0)dx
	+ \int_0^t\int_\Omega\na h_+(u,v)\cdot\na u dxds 
	+ \int_0^t\int_\Omega h_+(u,v)f'_\delta(u)dxds. \nonumber
\end{align}

It remains to estimate the last two integrals. For the last but one integral, we
insert the definition of $h_+(u,v)$ and apply Young's inequality:
\begin{align*}
  \int_0^t&\int_\Omega\na h_+(u,v)\cdot\na u dxds \\
	&= \int_0^t\int_\Omega\mathrm{1}_{\{0<u<1\}}\big((1-2u)h_0(v)|\na u|^2
	+ u(1-u)h'_0(v)\na v\cdot\na u\big)dxds \\
	&\le C\int_0^t\int_\Omega|\na u|^2 dxds + C\int_0^t\int_\Omega\mathrm{1}_{\{0<u<1\}}
	(1-u)|\na v|^2 dxds \\
	&\le C\int_0^t\int_\Omega|\na u|^2 dxds + C,
\end{align*}
where the last step follows from Lemma \ref{lem.estv}, and $C>0$ denotes here and in the following a constant independent of $\delta$
and $\kappa$. 

For the last integral in \eqref{2.aux2}, we observe that the function
$s\mapsto -s(1-s)(N^{-1}\log s-\log(1-s)+N^{-1}-1)$ is bounded in $[0,1]$.
We insert the definition of $f'_\delta(u)$ and distinguish three cases.
First, let $u\le\delta$. Then
\begin{align*}
  h_+(u,v)f'_\delta(u) &= [u]_+[1-u]_+\big(f'_\delta(\delta) 
	+ f''_1(\delta)(u-\delta) + \lambda(1-2u)\big) \\
	&= [u]_+[1-u]_+\bigg(\frac{1}{N}\log\delta - \log(1-\delta) + \frac{1}{N}-1\bigg) \\
	&\phantom{xx}{}+[u]_+[1-u]_+(u-\delta)\bigg(\frac{1}{N\delta} + \frac{1}{1-\delta}\bigg) 
	+ \lambda[u]_+[1-u]_+(1-2u) \\
	&\le \delta(1-\delta)\bigg(|\log(1-\delta)| + \frac{1}{N}\bigg)
	+ \lambda\delta(1-\delta) \le C,
\end{align*}
using $[u]_+[1-u]_+\le \delta(1-\delta)$ and $u-\delta\le 0$. 
Second, let $\delta<u<1-\delta$. We have
\begin{align*}
  h_+(u,v)f'_\delta(u) = u(1-u)\bigg(\frac{1}{N}\log u-\log(1-u)+\frac{1}{N}-1\bigg) 
	+ \lambda u(1-u)(1-2u) \le C,
\end{align*}
since $z\mapsto z\log z$ is bounded in $[0,1]$.
Finally, let $u\ge 1-\delta$ (and $\delta\le 1/2$). We obtain
\begin{align*}
  h_+(u,v)f'_\delta(u) &= [u]_+[1-u]_+\bigg(\frac{1}{N}\log(1-\delta) - \log\delta
	+ \frac{1}{N}-1\bigg) \\
	&\phantom{xx}{}+[u]_+[1-u]_+(u-\delta)\bigg(\frac{1}{N(1-\delta)}+\frac{1}{\delta}\bigg) 
	+ \lambda[u]_+[1-u]_+(1-2u) \\
	&\le \delta(1-\delta)\bigg(|\log\delta|+\frac{1}{N}\bigg) + (1-\delta)\delta(1-2\delta)
	\bigg(\frac{1}{N(1-\delta)} + \frac{1}{\delta}\bigg) + \lambda \le C.
\end{align*}
This proves that, for $0<t<T$,
$$
  \int_0^t\int_\Omega h_+(u,v)f'_\delta(u)dxds \le C(\Omega,T).
$$
Therefore, we infer from \eqref{2.aux2} that
\begin{align*}
  \frac12\int_\Omega&|\na u(t)|^2 dx + \int_\Omega f_\delta(u(t))dx
	+ \int_0^t\int_\Omega M_\delta(u)|\na\mu|^2 dxds \\
	&= \frac12\int_\Omega|\na u^0|^2 dx
	+ \int_\Omega f_\delta(u^0)dx + C\int_0^t\int_\Omega|\na u|^2 dxds + C.
\end{align*}
Since $u^0$ is strictly positive and bounded away from one, there exists $\delta_0>0$ such
that $f_\delta(u^0)=f(u^0)$ for $0<\delta\le\delta_0$. An application of Gronwall's lemma
shows \eqref{2.energy}.

{\em Step 2: Entropy inequality.} Because of the truncation, we have
$\na\Phi'_\delta(u)=\na u/M_\delta(u)\in L^2(\Omega_T)$, where $\Phi_\delta$
is defined in \eqref{2.Phi}. Thus, we can use $\phi_1=\Phi'_\delta(u)$ as a test function
in \eqref{2.appu}:
\begin{align}\label{2.Phid}
  \int_\Omega\Phi_\delta&(u(t))dx - \int_\Omega\Phi_\delta(u(0))dx
	= \int_0^t\langle\pa_t u,\Phi'_\delta(u)\rangle ds \\
	&= -\int_0^t\int_\Omega M_\delta(u)\na\mu\cdot\na\Phi'_\delta(u)dxds
	+ \int_0^t\int_\Omega h_+(u,v)\Phi'_\delta(u)dxds \nonumber \\
	&\le -\int_0^t\int_\Omega\na(-\Delta u+f'_\delta(u))\cdot\na u dxds
	+ \int_0^t\int_\Omega[u]_+[1-u]_+ h_0(v)\Phi'_\delta(u)dxds. \nonumber
\end{align}
The first integral on the right-hand side can be written as
$$
  -\int_0^t\int_\Omega\na(-\Delta u+f'_\delta(u))\cdot\na u dxds
	= -\int_0^t\int_\Omega(\Delta u)^2dxds - \int_0^t\int_\Omega f''_\delta(u)|\na u|^2 dxds.
$$
Because of $f''_{1,\delta}(u)\ge 0$ by \eqref{2.f1pp} and $f''_2(u)\ge -C$, we obtain
$$
  -\int_0^t\int_\Omega\na(-\Delta u+f'_\delta(u))\cdot\na u dxds
	\le -\int_0^t\int_\Omega(\Delta u)^2dxds + C\int_0^t\int_\Omega|\na u|^2 dxds.
$$
We claim that the integrand of the last integral in \eqref{2.Phid} is bounded, i.e.\
$[u]_+[1-u]_+\Phi'_\delta(u)$ is bounded uniformly in $u\in[0,1]$ and $\delta\in(0,1/2)$.
Indeed, if $\delta\le u\le 1-\delta$, we can compute
$$
  |[u]_+[1-u]_+\Phi'_\delta(u)| = \bigg|u(1-u)\int_{1/2}^u\frac{ds}{s(1-s)}\bigg|
	= \bigg|u(1-u)\log\frac{u}{1-u}\bigg| \le 1.
$$
If $0<u<\delta$, we find that
\begin{align*}
  |[u]_+[1-u]_+\Phi'_\delta(u)| &= \bigg|u(1-u)\bigg(\int_{1/2}^\delta\frac{ds}{s(1-s)}
	+ \int_\delta^u\frac{ds}{\delta(1-\delta)}\bigg)\bigg| \\
	&= u(1-u)\log\frac{\delta}{1-\delta} + u(1-u)\frac{\delta-u}{\delta(1-\delta)}.
\end{align*}
The first term is uniformly bounded since $|u\log\delta|\le |\delta\log\delta|\le 1$
and $|(1-u)\log(1-\delta)|\le 1$. This holds also true for the second term
because of $u(1-u)<\delta(1-\delta)$.
The final case $1-\delta<u<1$ is treated in a similar way:
\begin{align*}
  |[u]_+[1-u]_+\Phi'_\delta(u)| &= \bigg|u(1-u)\bigg(\int_{1/2}^{1-\delta}\frac{ds}{s(1-s)}
	+ \int_{1-\delta}^u\frac{ds}{\delta(1-\delta)}\bigg)\bigg| \\
	&= u(1-u)\log\frac{1-\delta}{\delta} + u(1-u)\frac{u-(1-\delta)}{\delta(1-\delta)}.
\end{align*}
The first term is uniformly bounded since $|(1-u)\log\delta|\le |\delta\log\delta|\le 1$
and $|u\log(1-\delta)|\le 1$, and the second term is bounded too.
We conclude from \eqref{2.Phid} that
\begin{align*}
  \int_\Omega\Phi_\delta&(u(t))dx + \int_0^t\int_\Omega(\Delta u)^2dxds
	\le \int_\Omega\Phi_\delta(u^0)dx + C\int_0^t\int_\Omega|\na u|^2 dxds,
\end{align*}
and the energy bound \eqref{2.energy} leads to \eqref{2.entropy}. 
\end{proof}

Finally, we derive a bound for the time derivatives of $u$ and $v$.

\begin{lemma}[Bounds for the time derivatives]\label{lem.timedelta}
There exists $C>0$ independent of $\delta$ and $\kappa$ such that
$$
  \|\pa_t u\|_{L^2(0,T;H^1(\Omega)')} + \|\pa_t v\|_{L^2(0,T;H^1(\Omega)')} \le C.
$$
\end{lemma}

\begin{proof}
The proof is similar to that one of Lemma \ref{lem.timeL}; we just have to estimate
the reaction terms. Since $0\le v\le 1$, we have the pointwise bounds
$g_+(u,v)=-[u]_+^1g_0(v)\le \max_{0\le v\le 1}g_0(v)$ and
$h_+(u,v)=[u]_+[1-u]_+h_0(v)\le \max_{0\le v\le 1}h_0(v)$.
Consequently, $\|g_+(u,$ $v)\|_{L^2(\Omega_T)}$ and $\|h_+(u,v)\|_{L^2(\Omega_T)}$ are uniformly bounded, concluding the proof.
\end{proof}


\section{The limit $(\delta,\kappa)\to 0$}\label{sec.limit}

Set $\kappa=\delta$ and let $(u_\delta,v_\delta,\mu_\delta)$ be a weak solution to 
\eqref{2.appu}--\eqref{2.appmu}.
Lemmas \ref{lem.estv}--\ref{lem.timedelta} give the following uniform bounds:
\begin{align}
  & 0\le v_\delta\le 1\quad\mbox{in }\Omega_T, \nonumber \\ 
  & \|D_+(u_\delta)^{1/2}\na v_\delta\|_{L^2(\Omega_T)} 
	+ \delta^{1/2}\|v_\delta\|_{L^2(0,T;H^1(\Omega))}
	+ \|\pa_t v_\delta\|_{L^2(0,T;H^1(\Omega)')} \le C, \nonumber \\ 
	& \|u_\delta\|_{L^\infty(0,T;H^1(\Omega))} 
	+ \|u_\delta\|_{L^2(0,T;H^2(\Omega))}
	+ \|\pa_t u_\delta\|_{L^2(0,T;H^1(\Omega)')} \le C, 
    \nonumber \\ 
	& \|M_\delta(u_\delta)\na\mu_\delta\|_{L^2(\Omega_T)} \le C. \nonumber 
\end{align}
The Aubin--Lions lemma \cite[Corollary 4]{Sim87} implies the existence of a subsequence, 
which is not relabeled, such that, as $\delta\to 0$,
$$
  u_\delta\to u\quad\mbox{strongly in }L^2(0,T;H^1(\Omega))\mbox{ and }C^0([0,T];L^2(\Omega)).
$$
We also have the weak convergences
\begin{align*}
  & v_\delta\rightharpoonup v \quad\mbox{weakly* in }L^\infty(0,T;L^\infty(\Omega)), \\
  & \pa_t u_\delta\rightharpoonup \pa_t u, \quad \pa_t v_\delta\rightharpoonup\pa_t v
	\quad\mbox{weakly in }L^2(0,T;H^1(\Omega)'), \\
	& D_+(u_\delta)\na v_\delta\rightharpoonup {\widetilde{I}}, 
	\quad M_\delta(u_\delta)\na\mu_\delta\rightharpoonup {\widetilde{J}}\quad\mbox{weakly in }L^2(\Omega_T), 
\end{align*}
where $\widetilde{I},\widetilde{J}\in L^2(\Omega_T)$, and it holds that $\delta\na v_\delta\to 0$ 
strongly in $L^2(\Omega_T)$. Before we identify the limits $\widetilde{I}$ and $\widetilde{J}$, we show that
the limit $u$ is bounded from below and above.

\begin{lemma}[$L^\infty$ bounds for $u$]\label{lem.Linftyu}
It holds that $0\le u\le 1$ in $\Omega_T$.
\end{lemma}

\begin{proof}
We proceed as in the proofs of \cite[Lemma 2]{ElGa96} or \cite[Theorem 5]{PePo21}.
Let $\alpha>0$ and introduce the set $V_{\alpha,\delta}=\{(x,t)\in\Omega_T:
u_\delta(x,t)\ge 1+\alpha\}$. Integrating $\Phi''_\delta(u_\delta(x,t))
=1/M_\delta(1-\delta)=1/(\delta(1-\delta))$ for $(x,t)\in V_{\alpha,\delta}$ twice gives
$$
  \Phi_\delta(u_\delta(x,t)) = \int_{1/2}^{u_\delta(x,t)}\int_{1/2}^s \frac{drds}{M_\delta(r)}
	= \frac{(u_\delta-1/2)^2}{2\delta(1-\delta)}\quad\mbox{for }(x,t)\in V_{\alpha,\delta}.
$$
The entropy estimate \eqref{2.entropy} shows that
$$
  \frac{\alpha^2|V_{\alpha,\delta}|}{2\delta(1-\delta)} 
	\le \int_{V_{\alpha,\delta}}\frac{(u_\delta-1/2)^2}{2\delta(1-\delta)} d(x,t)
	= \int_{V_{\alpha,\delta}}\Phi_\delta(u_\delta)d(x,t) \le C(T).
$$
Then we deduce from the a.e.\ pointwise limit $u_\delta(x,t)\to u(x,t)$ as $\delta\to 0$ 
{and} Fatou's lemma that
$$
  |\{u(x,t)\ge 1+\alpha\}| = \lim_{\delta\to 0}|V_{\alpha,\delta}|
	\le \lim_{\delta\to 0}\frac{2C(T)}{\alpha^2}\delta(1-\delta) = 0,
$$
implying that $u(x,t)\le 1+\alpha$ a.e.\ in $\Omega_T$
for all $\alpha>0$. Therefore, $u(x,t)\le 1$ in $\Omega_T$. 

A similar argument proves that $u\ge 0$ in $\Omega_T$. Indeed, let 
$W_{\alpha,\delta}=\{(x,t):u_\delta(x,t)\le-\alpha\}$ for $\alpha>0$. 
It follows from $\Phi''_\delta(u_\delta(x,t))=1/\delta(1-\delta)$ for $(x,t)\in W_{\alpha,\delta}$ that 
$\Phi_\delta(u_\delta(x,t))\le (1/2-u_\delta(x,t))^2/(2\delta(1-\delta))$. Hence,
$$
  \frac{\alpha^2|W_{\alpha,\delta}|}{2\delta(1-\delta)} 
	\le \int_{W_{\alpha,\delta}}\frac{(1/2-u_\delta)^2}{2\delta(1-\delta)} d(x,t)
	= \int_{W_{\alpha,\delta}}\Phi_\delta(u_\delta)d(x,t) \le C(T),
$$
and proceeding as before gives $|\{u(x,t)\le -\alpha\}|=0$ in the limit $\delta\to 0$ 
for all $\alpha>0$ and therefore $u\ge 0$ in $\Omega_T$.
\end{proof}

We continue by identifying $\widetilde{I}$. We conclude from
$[1-u_\delta]_+^1 v_\delta\rightharpoonup (1-u)v$ and 
$v_\delta\na u_\delta\rightharpoonup v\na u$ 
weakly in $L^2(\Omega_T)$ that
$$
  D_+(u_\delta)\na v_\delta = \na([1-u_\delta]_+^1v_\delta) 
	+ v_\delta\mathrm{1}_{\{0<u_\delta<1\}}\na u_\delta
	\rightharpoonup \na((1-u)v) + v\na u = (1-u)\na v
$$
weakly in $L^2(0,T;H^1(\Omega)')$. This shows that $\widetilde{I}=(1-u)\na v$ in $L^2(0,T;H^1(\Omega)')$. 

\begin{lemma}[Identification of $\widetilde{J}$]
It holds that $\widetilde{J}=-\na(M(u)\Delta u) + \na M(u)\Delta u+ M(u)\na f'(u)$ in the sense of $L^2(0,T;H^1(\Omega)')$
\end{lemma}

\begin{proof}
We proceed as in \cite[Section 3]{ElGa96}.
It holds for {$\phi\in C^\infty(\Omega_T)$ with $\na\phi\cdot\nu=0$ on $\pa\Omega$} that
\begin{align*}
  \int_0^T\int_\Omega & M_\delta(u_\delta)\na\mu_\delta\cdot\na\phi dxdt
	= \int_0^T\int_\Omega M_\delta(u_\delta)\na\big(-\Delta u_\delta
	+ f'_\delta(u_\delta)\big)\cdot\na\phi dxdt \\
	&= \int_0^T\int_\Omega M_\delta(u_\delta)\Delta u_\delta\Delta\phi dxdt
	+ \int_0^T\int_\Omega M'_\delta(u_\delta)\Delta u_\delta\na u_\delta\cdot\na\phi dxdt \\
	&\phantom{xx}{}+ \int_0^T\int_\Omega M_\delta(u_\delta)f''_\delta(u_\delta)
	\na u_\delta\cdot\na\phi dxdt =: J_1+J_2+J_3.
\end{align*}

First, we consider $J_1$. 
We observe that $M_\delta\to M$ uniformly, since by the mean-value theorem,
\begin{align*}
  |M_\delta(z)-M(z)| &\le \sup_{0<z<\delta}|M(\delta)-M(z)| 
	+ \sup_{1-\delta<z<1}|M(1-\delta)-M(z)| \\
	&\le M'(\xi_\delta)\delta + M'(\eta_\delta)\delta\to 0,
\end{align*}
where $\xi_\delta\in(z,\delta)$ and $\eta_\delta\in(1-\delta,z)$. This implies that $M_\delta(u_\delta)\to M(u)$ a.e.\ in $\Omega_T$
and, as $M_\delta$ is uniformly bounded, also strongly in $L^2(\Omega_T)$. Together
with the convergence $\Delta u_\delta\rightharpoonup\Delta u$ weakly in $L^2(\Omega_T)$,
we find that 
$$
  J_1 \to \int_0^T\int_\Omega M(u)\Delta u\Delta\phi dxdt. 
$$

For the integral $J_2$, we claim that $M'_\delta(u_\delta)\na u_\delta\to M'(u)\na u$
strongly in $L^2(\Omega_T)$. This limit is not trivial since $M'_\delta$ is discontinuous
at $\delta$ and $1-\delta$. We consider the integrals
\begin{align*}
  &\int_0^T\int_\Omega|M'_\delta(u_\delta)\na u_\delta-M'(u)\na u|^2 dxdt
	= \int_0^T\int_{\{0<u<1\}}|M'_\delta(u_\delta)\na u_\delta-M'(u)\na u|^2 dxdt \\
	&{}+ \int_0^T\int_{\{u=0\}}
	|M'_\delta(u_\delta)\na u_\delta-M'(u)\na u|^2 dxdt + \int_0^T\int_{\{u=1\}}
	|M'_\delta(u_\delta)\na u_\delta-M'(u)\na u|^2 dxdt.
\end{align*}
On the set $\{0<u<1\}$, we know that $M'_\delta(u_\delta)\to M'(u)$ a.e.\ in $\Omega_T$
and, because of the strong convergence of $(\na u_\delta)$, also
$M'_\delta(u_\delta)\na u_\delta\to M'(u)\na u$ a.e.\ in $\Omega_T$ (possibly for a subsequence). 
Moreover, $|M'_\delta(u_\delta)\na u_\delta|^2$ is uniformly bounded on $\{0<u<1\}$.
Therefore, by dominated convergence, 
$$
  \int_0^T\int_{\{0<u<1\}}|M'_\delta(u_\delta)\na u_\delta-M'(u)\na u|^2 dxdt\to 0.
$$
It follows from $\na u=0$ on $\{u=0\}\cup\{u=1\}$ and the uniform bound for
$M'_\delta$ that
\begin{align*}
  \int_0^T\int_{\{u=0\}}&|M'_\delta(u_\delta)\na u_\delta-M'(u)\na u|^2 dxdt
	= \int_0^T\int_{\{u=0\}}|M'_\delta(u_\delta)\na u_\delta|^2 dxdt \\
	&\le C\int_0^T\int_{\{u=0\}}|\na u_\delta|^2dxdt
	\to \int_0^T\int_{\{u=0\}}|\na u|^2 dxdt = 0.
\end{align*}
The limit in the remaining integral over $\{u=1\}$ vanishes in the same way.
This shows that
$$
  J_2 \to \int_0^T\int_\Omega M'(u)\Delta u\na u\cdot\na\phi dxdt.
$$

Finally, for the limit in $J_3$, we observe that
$M_\delta(z)f''_\delta(z) = M_\delta(z)(f''_{1,\delta}(z)+f''_2(z))$ is uniformly bounded,
since the singularities as $\delta\to 0$ in $f_{1,\delta}''$ are canceled by the factor 
$M_\delta(z)$. Thus, it remains to show that $M_\delta(u_\delta)f''_\delta(u_\delta)\to
M(u)f''(u)$ in $\Omega_T\setminus N$, where $N$ is a set of measure zero. 
To this end, we distinguish several cases.

Let $(x,t)\in\Omega_T\setminus N$ and $0<u(x,t)<1$. 
For given $\eps>0$, there exists $0<\delta<\eps$ 
such that $\delta<\eps\le u_\delta(x,t)\le 1-\eps<1-\delta$. At this point, we have
$M_\delta(u_\delta(x,t))f''_\delta(u_\delta(x,t))=M(u_\delta(x,t))f''(u_\delta(x,t))
\to M(u(x,t))f''(u(x,t))$. Next, if $u(x,t)=1$, we choose $\delta>0$ such that
$u_\delta(x,t)\ge 1-\delta$. Then
\begin{align*}
  M_\delta&(u_\delta(x,t))f''_\delta(u_\delta(x,t))
	= M(\delta)(f''_1(\delta)+f_2(u_\delta)) \\
	&= N^{-1}\delta + (1-\delta) + \delta(1-\delta)f_2(u_\delta) \to 1 = (Mf'')(1).
\end{align*}
On the other hand, if $u_\delta(x,t)<1-\delta$ and $u_\delta(x,t)\to 1$,
\begin{align*}
  M_\delta&(u_\delta(x,t))f''_\delta(u_\delta(x,t))
	= M(u_\delta(x,t))f''(u_\delta(x,t)) \\
	&= N^{-1}(1-u_\delta(x,t)) + u_\delta(x,t) + u_\delta(1-u_\delta)f''_(u_\delta)
	\to 1 = (Mf'')(1).
\end{align*}
The case $u(x,t)=0$ is treated in a similar way.
We conclude that $M_\delta(u_\delta)f''_\delta(u_\delta)\to M(u)f''(u)$ strongly
in $L^2(\Omega_T)$. Then, in view of the strong convergence of $(\na u_\delta)$,
$$
  J_3 \to \int_0^T\int_\Omega M(u)f''(u)\na u\cdot\na\phi dxdt.
$$

Summarizing, we have shown that
\begin{align*}
  \int_0^T\int_\Omega M_\delta(u_\delta)\na\mu_\delta\cdot\na\phi dxdt
	&\to \int_0^T\int_\Omega\big(M(u)\Delta u\Delta\phi + M'(u)\Delta u\na u\cdot\na\phi \\
	&\phantom{xx}{}+ M(u) f''(u)\na u\cdot\na\phi\big) dxdt,
\end{align*}
and the right-hand side can be identified as the weak formulation of $\widetilde{J}$.
\end{proof}

\begin{remark}\rm 
Choosing the mobility such that $\Phi(0)=\Phi(1)=\infty$, one can show that 
$\{u=0\}\cup\{u=1\}$ has measure zero, which means that $0<u<1$ holds a.e.\ in $\Omega_T$, 
and we can write $J=M(u)\na(-\Delta u+f'(u))$ in the sense of distributions. 
The claim that $\{u=0\}\cup\{u=1\}$ has measure zero 
can be proved as in {\cite[Corollary, p.~417]{ElGa96}}. It follows from the entropy
bound $\int_\Omega\Phi_\delta(u_\delta(t))dx\le C(T)$ and the fact that
$\liminf_{\delta\to 0}\Phi_\delta(u_\delta)=\Phi(u)$ if $0<u<1$ and
$\liminf_{\delta\to 0}\Phi_\delta(u_\delta)=\infty$ else.
\end{remark}

It remains to pass to the limit $\delta\to 0$ in the reaction terms.
Since $(v_\delta)$ is only converging weakly, this limit is not trivial.
The idea is to use the Browder--Minty trick, which is possible since
$(u_\delta)$ converges strongly in $L^2(0,T;H^1(\Omega))$. 

\begin{lemma}\label{lem.gh}
It holds that $g_+(u_\delta,v_\delta)\rightharpoonup g(u,v)$ and 
$h_+(u_\delta,v_\delta)\rightharpoonup h(u,v)$ weakly in $L^2(\Omega_T)$ as $\delta\to 0$.
\end{lemma}

\begin{proof}
We only show the limit in $h_+(u_\delta,v_\delta)$ as the proof in $g_+(u_\delta,v_\delta)$
is similar. We know that $(\pa_t v_\delta)$ is bounded in $L^2(0,T;H^1(\Omega)')$ and
$(v_\delta)$ is bounded in $L^2(\Omega_T)$. Since the embedding
$L^2(\Omega)\hookrightarrow H^1(\Omega)'$ is compact, we infer from the Aubin--Lions lemma
that, up to a subsequence, $v_\delta\to v$ strongly in $L^2(0,T;H^1(\Omega)')$. 
Moreover, $([1-u_\delta]_+^{1/2}\na v_\delta)$ is bounded in $L^2(\Omega_T)$.
Furthermore, we know that $(u_\delta)$ is bounded in $L^\infty(0,T;H^1(\Omega))$
and $L^2(0,T;H^2(\Omega))$, and $u_\delta\to u$ strongly in $L^2(0,T;H^1(\Omega))$.

Let $y\in C_0^\infty(\Omega_T)$. It follows from the monotonicity of $h_0$ that
\begin{align}\label{minty}
  0 &\le \int_0^T\int_\Omega [u_\delta]_+[1-u_\delta]_+(v_\delta-y)(h_0(v_\delta)-h_0(y))dxdt \\
	&= \int_0^T\big\langle v_\delta-y,[u_\delta]_+[1-u_\delta]_+
	(h_0(v_\delta)-h_0(y)\big\rangle dt, \nonumber
\end{align}
recalling that $\langle\cdot,\cdot\rangle$ is the dual product between $H^1(\Omega)'$ and
$H^1(\Omega)$. 
This formulation is possible if $[u_\delta]_+[1-u_\delta]_+h_0(v_\delta)\in 
L^2(0,T;H^1(\Omega))$. To verify this statement, we observe that
$\na u_\delta\in L^2(0,T;H^1(\Omega))$ 
implies that $(1-2u_\delta)\mathrm{1}_{\{0<u_\delta<1\}}\na u_\delta\in L^2(0,T;L^2(\Omega))$. 
Moreover, $[u_\delta]_+[1-u_\delta]_+^{1/2}\in L^\infty(\Omega_T)$
and $[1-u_\delta]_+^{1/2}\na v_\delta\in L^2(\Omega_T)$. This shows that
$$
  \na\big([u_\delta]_+[1-u_\delta]_+h_0(v_\delta)\big)
	= [u_\delta]_+[1-u_\delta]_+h'_0(v_\delta)\na v_\delta
	+ h_0(v_\delta)(1-2u_\delta)\mathrm{1}_{\{0<u_\delta<1\}}\na u_\delta
$$
is a function in $L^2(\Omega_T)$,
so that $[u_\delta]_+[1-u_\delta]_+h_0(v_\delta)\in L^2(0,T;H^1(\Omega))$. 

Let $h_1$ be the weak* limit of $(h_0(v_\delta))$ in $L^\infty(0,T;L^\infty(\Omega))$
and $h_2$ be the weak limit of $([u_\delta]_+[1-u_\delta]_+h_0(v_\delta))$ in $L^2(\Omega_T)$.
We claim that $h_2=u(1-u)h_1$. Indeed, since $(u_\delta)$ converges strongly in
$L^2(0,T;H^1(\Omega))$, $[u_\delta]_+[1-u_\delta]_+h_0(v_\delta)\rightharpoonup u(1-u)h_1$
weakly in $L^2(\Omega_T)$ (here, we use $0\le u\le 1$ in $\Omega_T)$; see Lemma
\ref{lem.Linftyu}), and we deduce from the uniqueness of the limit that $u(1-u)h_1=h_2$. 

We can now pass to the limit $\delta\to 0$ in \eqref{minty} to find that
$$
  0\le \int_0^T\big\langle v-y,u(1-u)(h_1-h_0(y)\big\rangle dt
	= \int_0^T\int_\Omega u(1-u)(h_1-h_0(y))(v-y)dxdt.
$$
By density, this inequality holds for all $y\in L^2(\Omega_T)$. 
Let $w\in L^2(\Omega_T)$ and choose $y=v-\eta w$ for $\eta\in\R$. Then
$$
  0\le \eta\int_0^T\int_\Omega u(1-u)(h_1-h_0(v-\eta w))wdxdt.
$$
Choosing $\eta>0$ and performing the limit $\eta\to 0$ yields
$\int_0^T\int_\Omega u(1-u)(h_1-h_0(v))w dxdt\ge 0$. On the other hand, if
$\eta<0$ and $\eta\to 0$, we have $\int_0^T\int_\Omega u(1-u)(h_1-h_0(v))w dxdt\le 0$.
Since $w$ is arbitrary, $u(1-u)h_1 = u(1-u)h_0(v)$. Thus, 
$$
  h_+(u_\delta,v_\delta) = [u_\delta]_+[1-u_\delta]_+ h_0(v_\delta) 
	\rightharpoonup u(1-u)h_0(v)\quad\mbox{weakly in }L^2(\Omega_T).
$$
This ends the proof.
\end{proof}

\begin{remark}[Generalizations]\label{rem.gener}\rm
It is possible to generalize the relations \eqref{1.M} and \eqref{1.gh}
for the mobility and the reaction rates. For instance, we may choose
$M(u)=u^m(1-u)^mM_0(u)$ for $m\ge 1$ and $0<m_*\le M_{0}(u)\le m^*$ for $u\in[0,1]$,
where $m^*\ge m_*>0$; see \cite{ElGa96}. In fact, we just need $M(0)=M(1)=0$
and $M(u)f''(u)\in C^0([0,1])$; see \cite{PePo21}. The latter condition is needed
to identify the weak limit $J$. The reaction terms may be generalized to
$g(u,v)=g_0(v)g_1(u)$ and $h(u,v)=h_0(v)h_1(u)$, for instance, where we assume that
$g_1$ is bounded in $[0,1]$; $g_0$ grows at most linearly;
$h_1$ satisfies $h_1(u)f'(u)\le C$ for all $u\in[0,1]$ to cancel the singularities
of $f'$; and $|h_1(u)|\le C(1-u)$ for $u\in[0,1]$ to estimate in Step 1
of the proof of Lemma \ref{lem.energy} the integral
$$
  \int_\Omega h_1(u)h'_0(v)\na v\cdot\na u dx \le \int_\Omega|\na u|^2 dx
	+ C\int_\Omega (1-u)|\na v|^2 dx.
$$
Clearly, also the free energy $f(u)$ may be generalized if the
factors in the diffusion and reaction terms are adapted in such a way that the singularities from $f'(u)$ are canceled.
\end{remark}


\section{Numerical experiments}\label{sec.num}

\subsection{Scaling of the equations}\label{sec.scaling}

The biofilm model with physical units reads as follows:
\begin{align*}
  & \pa_t v - \diver(D(1-u)\na v) = -R_c uv, \\
  & \pa_t u - \diver(M'u(1-u)\na\mu) = u(1-u)\frac{R_p v}{K_v+v}, \\
  & \mu = -\Gamma_1\Delta u + \Gamma_2 f'(u),
\end{align*}
and $f'(u)$ is given by \eqref{1.FH}, observing that the
parameters $N$ and $\lambda$ and the volume fraction $u$ 
are dimensionless. Here, $D>0$ is the diffusivity, $M'>0$ the mobility
constant, $R_c>0$ the consumption rate, $R_p>0$ the production rate,
$\Gamma_1>0$ the parameter of the distortional energy, and $\Gamma_2>0$
the parameter of the mixing free energy.

Choosing the characteristic length $x_0$,
the characteristic time $t_0$, the characteristic concentration
$v_0$, and the characteristic chemical potential $\mu_0$, the
scaled equations read as follows:
\begin{align}
  & \pa_t v - \diver(D_0(1-u)\na v) = -R_c^0 uv, \label{5.v} \\
  & \pa_t u - \diver(M_0u(1-u)\na\mu) = u(1-u)\frac{R_p^0 v}{K+v}, 
  \label{5.u} \\
  & \mu = -\Gamma_1^0\Delta u + \Gamma_2^0 f'(u), \label{5.mu}
\end{align}
where the dimensionless parameters are
\begin{align*}
  & D_0 = \frac{Dt_0}{x_0^2}, \quad M_0 = \frac{M't_0\mu_0}{x_0^2}, 
  \quad R_c^0 = R_ct_0, \quad R_p^0 = R_pt_0, \\
  & K = \frac{K_v}{v_0}, \quad 
  \Gamma_1^0 = \frac{\Gamma_1}{\mu_0x_0^2}, \quad 
  \Gamma_2^0 = \frac{\Gamma_2}{\mu_0}.
\end{align*}
The model of \cite{WaZh12} (without elastic energy contributions) reads as
\begin{align*}
  & \pa_t((1-u)v) - \diver(D_0(1-u)\na v) 
  = -u\frac{R_c v}{\widetilde{K}+v}, \\
  & \pa_t u - \diver(M_0(1-u)\na\mu) 
  = u\frac{R_p v}{K_v+v}, \\
  & \mu = -\Gamma_1\Delta u + \Gamma_2 f'(u).
\end{align*}

\begin{table}[ht]
\centering
\begin{tabular}{|c|l|c|l|}
\hline
Symbol & Parameter & Value & Unit \\ 
\hline 
$D$      & Diffusivity      & $10^{-10}$ & m$^2$\,s$^{-1}$ \\
$M'$     & Mobility         & $2.5\cdot 10^{-8}$ & s \\
$R_c$    & Consumption rate & $10^{-2}$ & s$^{-1}$ \\
$R_p$    & Production rate  & $10^{-2}$ & kg\,m$^{-3}$\,s$^{-1}$ \\
$K_v$    & Half-saturation constant & $10^{-4}$ & kg\,m$^{-3}$ \\
$\Gamma_1$ & Distortional energy & $4\cdot 10^{-15}$ & m$^4$\,s$^{-2}$ \\
$\Gamma_2$ & Mixing free energy & $4\cdot 10^{-6}$ & m$^{2}$\,s$^{-2}$ \\
$N$ & Polymerization parameter  & $10^{3}$ & \\
$\lambda$ & Flory--Huggins parameter & 0.55 & \\
$x_0$    & Characteristic length & $10^{-4}$ & m \\
$t_0$    & Characteristic time   & $10^2$ & s \\
$v_0$    & Characteristic concentration & $10^{-3}$ & kg\,m$^{-3}$ \\
$k_BT$   & Thermal energy at $T=300$\,K & $4\cdot 10^{-21}$ & 
kg\,m$^2$\,s$^{-2}$ \\
$\widetilde{K}$ & Half-saturation constant for model of
\cite{WaZh12} & $5\cdot 10^{-4}$ &  \\
\hline
\end{tabular}

\medskip
\caption{Parameters used in the numerical simulations.}
\label{table}
\end{table}

The characteristic chemical potential $\mu_0$ is determined
by the thermal energy and the characteristic concentration
and length (see Table \ref{table}) as 
$\mu_0 = k_BT/(v_0x_0^3) = 4\cdot 10^{-6}$\,m$^2$s$^{-2}$.
The values of the physical parameters in Table \ref{table} differ
from those in \cite{WaZh12} but are of a similar order. With our values, the scaled parameters are of order one (except $K$ and $\Gamma_1^0$):
$$
  D_0 = R_c^0 = R_p^0 = 1, \quad K = 10^{-1}, \quad M_0 = 10^{-3},
  \quad\Gamma_1^0 = 10^{-1}, \quad \Gamma_2^0 = 1.
$$

\subsection{Numerical discretization}

As in \cite{ZCW08b}, we approximate equations \eqref{5.v}--\eqref{5.mu} in the one-dimensional domain $\Omega=(0,1)$ by a BDF2 (second-order Backward Differentiation Formula) discretization in time. The spatial discretization is performed by finite volumes. The scheme is explicit for the mobility and potential, using the second-order approximation $\bar{u}^k:=2u^{k-1}-u^{k-2}$, but implicit in the reactions and semi-implicit in the diffusion. Let $\Delta t>0$ be the time step size, $\Delta x>0$ the space grid size, and $x_i=i\Delta x$, $x_{i\pm 1/2}=(i\pm1/2)\Delta x$. We introduce finite-volume cells $K_i = (x_{i-1/2},x_{i+1/2})$ for $i = 1,\ldots,N_x$. Then the values $u_i^k$, $v_i^k$, and $\mu_i^k$ approximate $u(x_i,k \Delta t)$, $v(x_i, k \Delta t)$, and $\mu(x_i, k\Delta t)$ respectively for $i = 1,\ldots,N_x$, $k = 1,\ldots,N_T$. Our scheme reads for $k\ge 2$ as follows:
\begin{align*}
  & \frac{\Delta x}{2\Delta t}(3v_i^k-4v_i^{k-1}+v_i^{k-2})
  + \mathcal{G}_{i+1/2}^k-\mathcal{G}_{i-1/2}^k
  = -\Delta x R_c^0 u_i^k v_i^k, \\ 
  & \frac{\Delta x}{2\Delta t}(3u_i^k-4u_i^{k-1}+u_i^{k-2})
  + \mathcal{F}_{i+1/2}^k-\mathcal{F}_{i-1/2}^k
  = \Delta x u_i^k(1-u_i^k)\frac{R_p^0 v_i^k}{K+v_i^k}, \\
  & \mathcal{H}_{i+1/2}^k-\mathcal{H}_{i-1/2}^k + \Delta x f'(\bar{u}_i^k)
  = \Delta x\mu_i^k,
\end{align*}
where the numerical fluxes are given by
\begin{align*}
  & \mathcal{G}_{i+1/2}^k = -D_0(1-u_{i+1/2}^k)
  \frac{v_{i+1}^k-v_i^k}{\Delta x}, \\
  & \mathcal{F}_{i+1/2}^k = -M_0u_{i+1/2}^k(1-u_{i+1/2}^k)
  \frac{\mu_{i+1}^k-\mu_i^k}{\Delta x}, \quad
  \mathcal{H}_{i+1/2}^k = -\frac{u_{i+1}^k-u_i^k}{\Delta x},
\end{align*}
and $u_{i+1/2}^k=\frac12(u_{i+1}^k+u_i^k)$. The approximation $(u_i^1,v_i^1,\mu_i^1)$ at the first time step is computed from the implicit Euler method. 

In the same way, we discretized a simplified version of \cite{WaZh12} which reads in its dimensionless form for $k \geq 2$ as
\begin{align*}
  & \frac{\Delta x}{2\Delta t}(3w_i^k-4w_i^{k-1}+w_i^{k-2})
  + \mathcal{G}_{i+1/2}^k-\mathcal{G}_{i-1/2}^k
  = -\Delta x  u_i^k \frac{\widetilde{R}_c^0 v_i^k}{\tilde{K}+v}, \\ 
  & \frac{\Delta x}{2\Delta t}(3u_i^k-4u_i^{k-1}+u_i^{k-2})
  + \widetilde{\mathcal{F}}_{i+1/2}^k-\widetilde{\mathcal{F}}_{i-1/2}^k
  = \Delta x u_i^k\frac{R_p^0 v_i^k}{K+v_i^k}, \\
  & \mathcal{H}_{i+1/2}^k-\mathcal{H}_{i-1/2}^k + \Delta x f'(\bar{u}_i^k)
  = \Delta x\mu_i^k,
\end{align*}
where we abbreviated $w_i^k=(1-u_i^k)v_i^k$,
$\mathcal{G}$ and $\mathcal{H}$ are as above, 
$\widetilde{\mathcal{F}} = -M_0 u_{i+1/2}^k (\mu_{i+1}^k-\mu_i^k)/\Delta x$, and $\widetilde{R}_c^0=1$, $R_p^0=1$ are scaled rates. We use the Newton method to solve the resulting system of nonlinear equations. For the first three test cases, we used a mesh of $128$ cells and the time step size $\Delta t=10^{-3}$.

\subsection{Numerical results}

\subsubsection*{Test case $1$:}
We consider the initial conditions 
\begin{align*}
	u^0(x) = \frac{1}{2} \sin(2 \pi x)^2 + 2\cdot 10^{-2},  \quad
	v^0(x) \equiv 0.75. 
\end{align*}
The numerical solutions $u$ and $v$ are presented in Figure \ref{fig.1}.
The substrate concentration converges uniformly to zero as $t\to\infty$ because of the consumption term, while the volume fraction of the biomass is increasing in time. The increase becomes slower and stops after some time since the production term is proportional to the substrate concentration which almost vanishes for large times and hence the production term vanishes too.
In our model, both the biomass fraction and the substrate concentration change at a slower rate compared to the model of \cite{WaZh12}, which is caused by the additional factor $1-u$ in the source term. Accordingly, the convergence to the steady state is smaller in our model than in the model of \cite{WaZh12}. Note that, without the additional factor $1-u$, an initial value $u^0$ smaller but close to one may lead to a volume fraction exceeding its maximal value and consequently break down the numerical scheme.

\begin{figure}[ht]
	\includegraphics[scale=0.26]{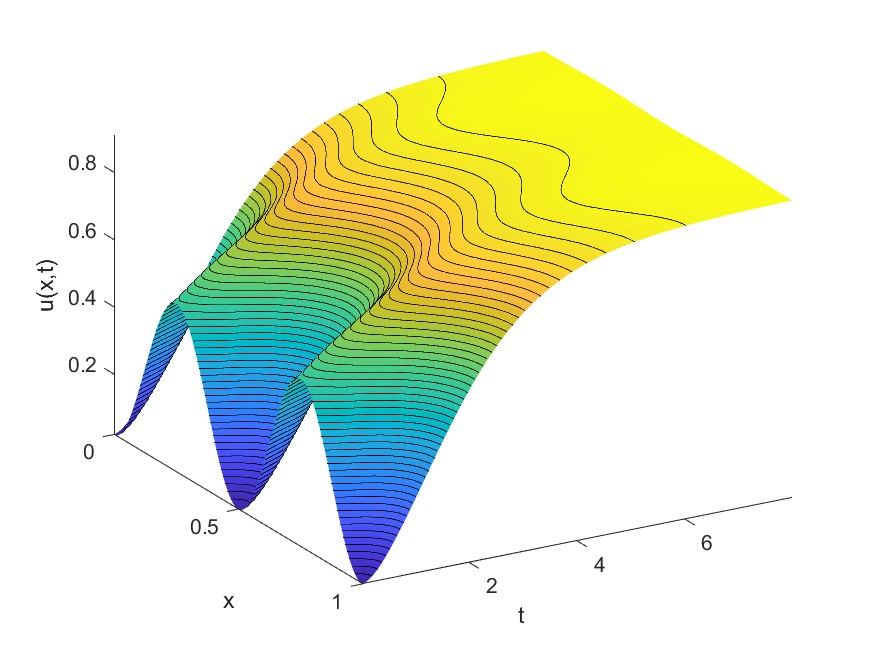}
	\includegraphics[scale=0.26]{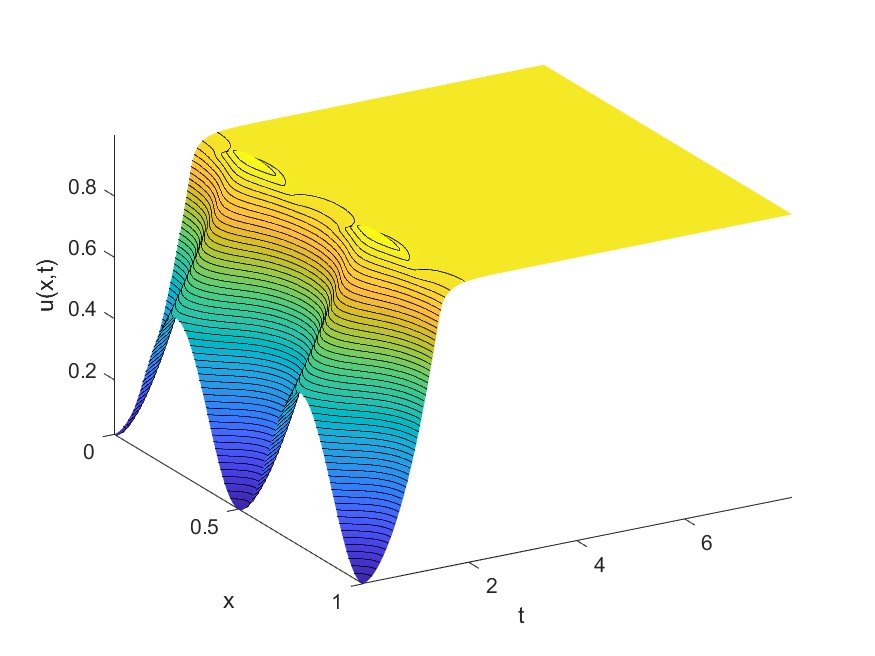}
	\includegraphics[scale=0.26]{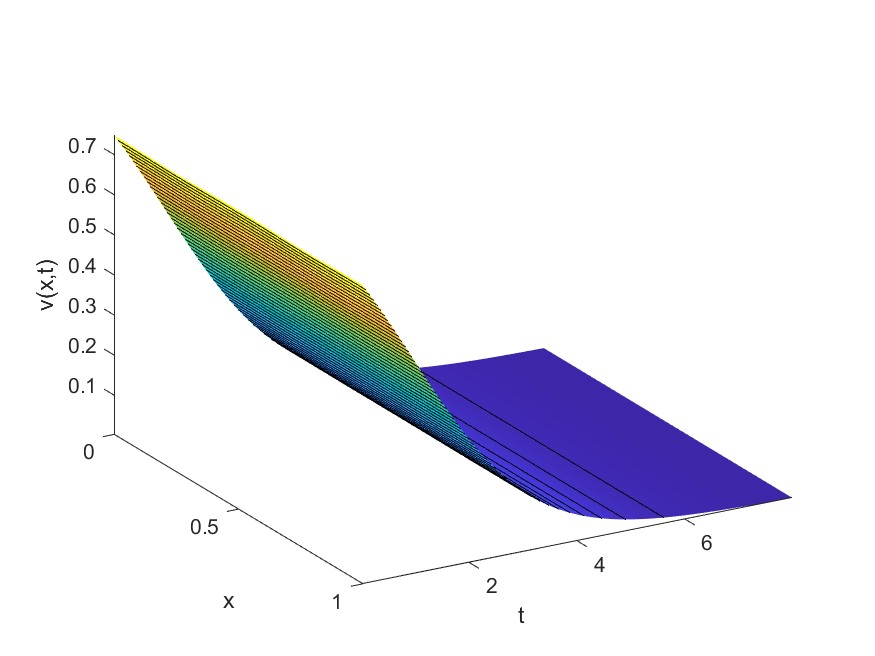}
	\includegraphics[scale=0.26]{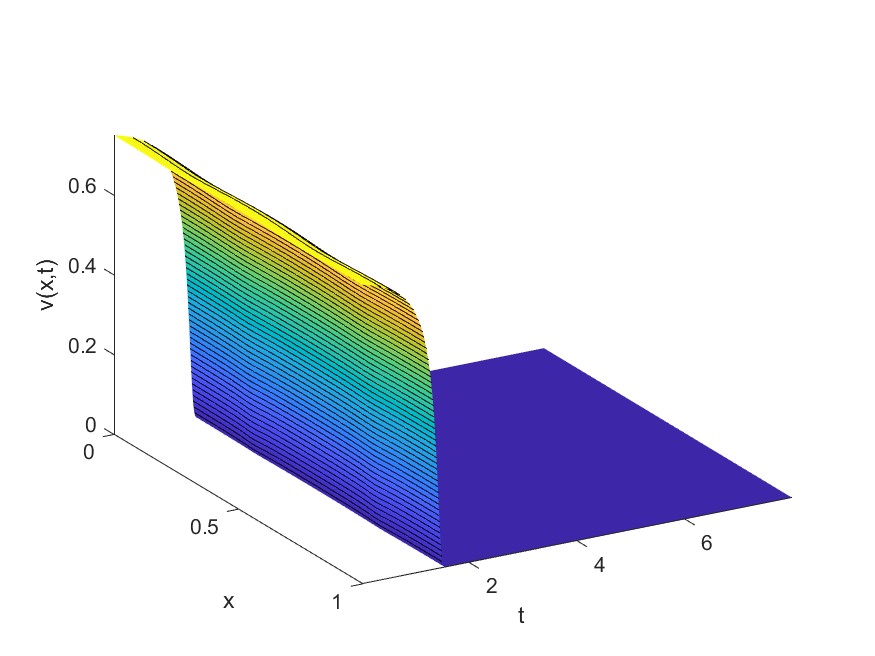}
	\caption{Biomass fraction $u$ (top) and substrate concentration $v$ (bottom) in test case $1$ for our system (left) and the system of \cite{WaZh12} (right).}
    \label{fig.1}
\end{figure}


\subsubsection*{Test case $2$:} 
We consider the initial conditions 
\begin{align*}
	u^0(x) = \begin{cases} 0.2, \quad &\text{ if } 0 \leq x \leq 0.2, \\
	1 \cdot 10^{-2}, \quad &\text{ if } 0.2 < x \leq 1, 
	\end{cases} \quad v^0(x) \equiv 0.1.
\end{align*}
In both models, the volume fraction of biomass growths rather fast until the substrate concentration vanishes; see Figure \ref{fig.2}.
Due to the additional factor $1-u$ in our mobility, we can observe a slower diffusion in areas of larger volume fraction compared to \cite{WaZh12}. In areas of low volume fractions, we observe a larger growth than for \cite{WaZh12}, which can be explained by the larger nutrient consumption compared to our model, causing a lack of nutrient supply for further growth.

\begin{figure}[ht]
	\includegraphics[scale=0.26]{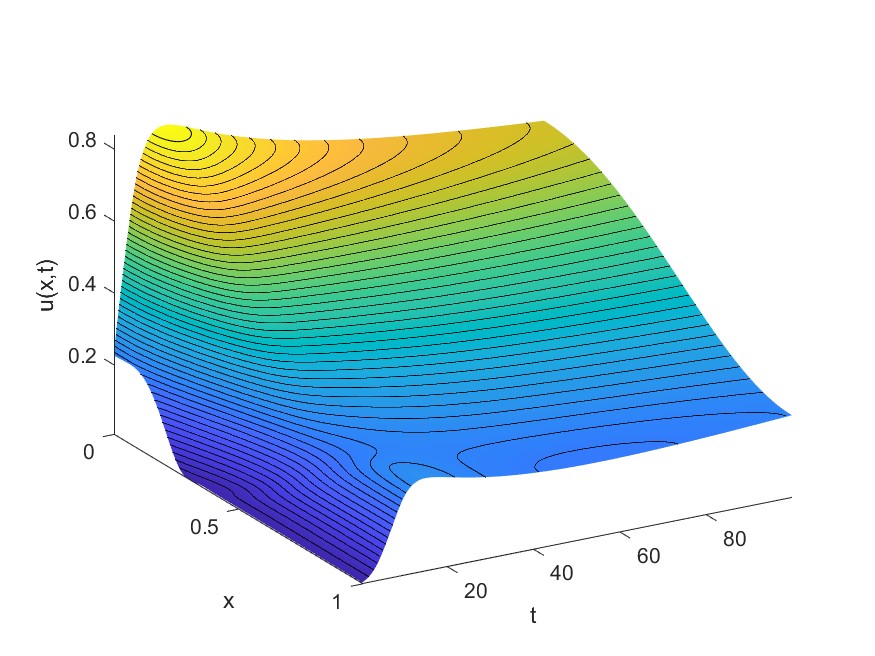}
	\includegraphics[scale=0.26]{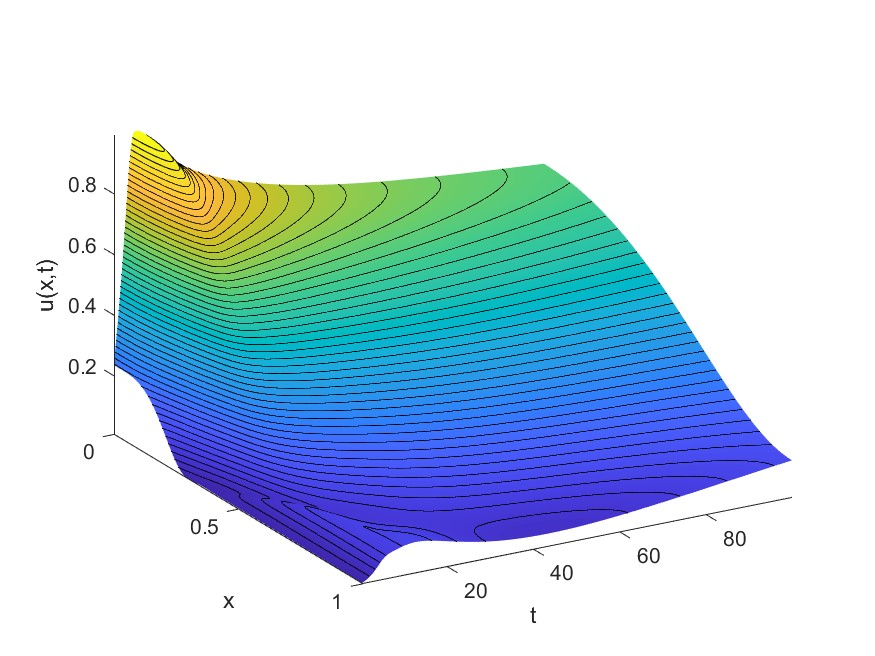}
	\caption{Biomass $u$ in test case $2$ for our system (left) and the system of \cite{WaZh12} (right).}
    \label{fig.2}
\end{figure}


\subsubsection*{Test case $3$:}
We choose the initial conditions 
\begin{equation}\label{5.ic}
	u^0(x) = -(x-1/2)^2+1/3, \quad
	v^0(x) \equiv 0.3.
\end{equation}
As in the previous test cases, we observe in Figure \ref{fig.3} a faster growth of biomass volume fraction in the model of \cite{WaZh12}. Moreover, the growth process dominates before the diffusion process flattens the maximal volume fraction towards the steady state. Due to the absence of the factor $1-u$, this effect is stronger than in the model of \cite{WaZh12}.

\begin{figure}[ht]
	\includegraphics[scale=0.26]{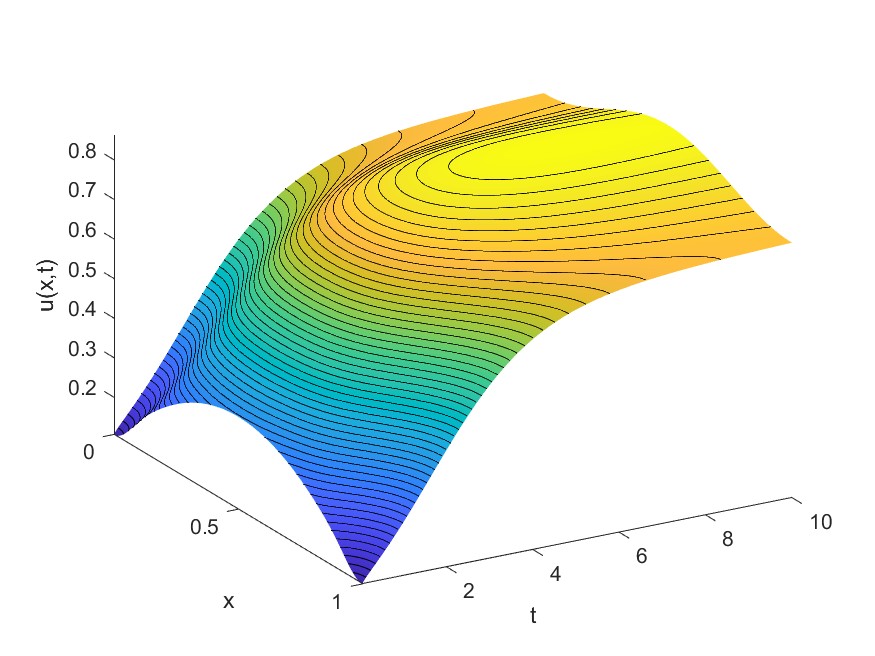}
	\includegraphics[scale=0.26]{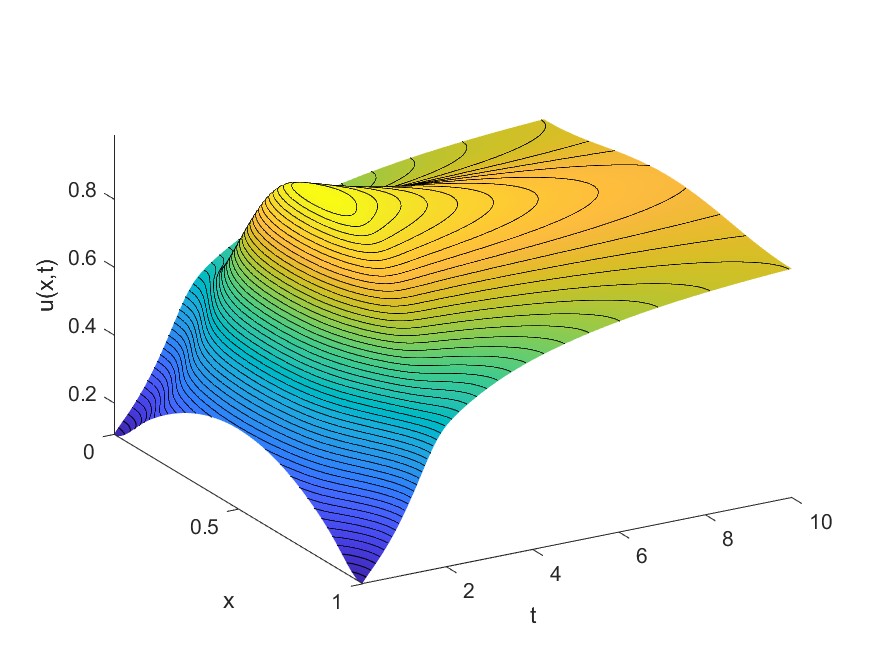}
	\caption{Biomass $u$ in test case $3$ for our system (left) and the system of \cite{WaZh12} (right).}
    \label{fig.3}
\end{figure}


\subsubsection*{Test case $4$:}
We analyze the order of convergence in space with the initial conditions
\eqref{5.ic}. Since there does not exist an explicit solution, we compute a reference solution $(u_{\mathrm{ref}},v_{\mathrm{ref}})$ at time $T = 1$ on a mesh with $2048$ cells with time step size  $\Delta t = 10^{-5}$. The approximate solutions $u^{(j)}$ are determined on meshes of $2^j$ cells for $j = 4,\ldots,10$.
We choose a rather small value for $T$ to compute the order of convergence in space before a steady state is reached.
Figure \ref{fig.4} (left) illustrates the discrete $L^2$ norm of the difference $u_{\rm ref}-u^{(j)}$ for $j=4,\ldots,10$. As expected, we observe a second-order convergence in space.

\begin{figure}[ht]
	\includegraphics[scale=0.26]{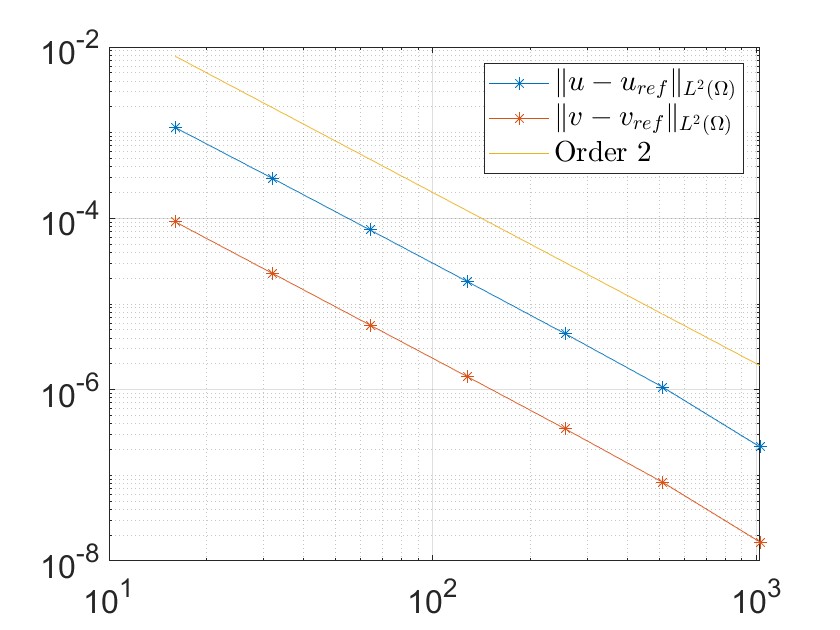}
    \includegraphics[scale=0.26]{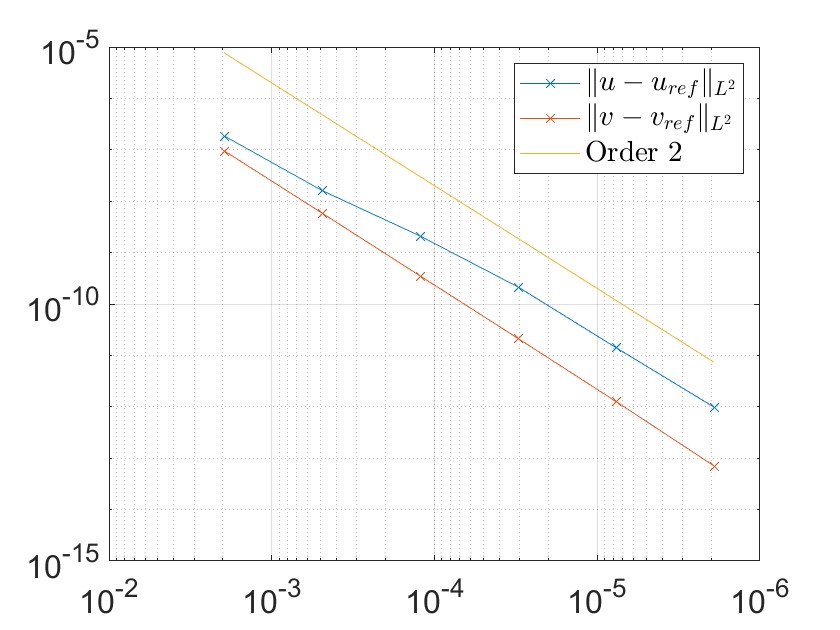}
	\caption{Convergence in space (left) and
    convergence in time (right) at time $T = 1$.}
    \label{fig.4}
\end{figure}

\subsubsection*{Test case $5$:}
We analyze the order of convergence in time by using as before the initial conditions \eqref{5.ic} and by choosing $L=128$ cells in space.
We compute a reference solution $(u_{\mathrm{ref}},v_{\mathrm{ref}})$ 
at time $T = 1$ with time step size $\Delta t = 1/(2^{14} L) \approx 5 \cdot 10^{-7}$. The approximate solutions $u^{(j)}$ are determined with time step sizes $\Delta t = 1/{(2^{2j}L)}$ for $j = 1,\ldots,6$. 
Figure \ref{fig.4} (right) illustrates the discrete $L^2$ norm of the difference $u_{\rm ref}-u^{(j)}$ for $j = 1,\ldots,6$.
We observe a convergence in time of order $1.73$ for $u$ and $2$ for $v$, respectively.


\end{document}